\documentclass[11pt]{amsart}\usepackage{tikz-cd}
\usepackage{mathtools}
\usepackage{amsfonts}
\usepackage{graphicx}
\usepackage{verbatim}
\usepackage{textcomp}\usepackage{amssymb}
\usepackage{cite}
\usepackage{color}
\usepackage[all]{xy}
\usepackage{graphicx}
\usepackage{color}

\definecolor{NoteColor}{rgb}{1,0,0}

\newcommand{\note}[1]{\textcolor{NoteColor}{^\sharp 1}}

\usepackage{hyperref}
\usepackage{soul}
\hypersetup{
    colorlinks,
    citecolor=black,
    filecolor=black,
    linkcolor=black,
    urlcolor=black
}
\usepackage{tikz}
\usepackage[all]{xy}
\usepackage{graphicx}
\usetikzlibrary{backgrounds}
\newtheorem{theorem}{Theorem}[section]
\newtheorem{lemma}[theorem]{Lemma}
\newtheorem{proposition}[theorem]{Proposition}
\newtheorem{corollary}[theorem]{Corollary}   

\allowdisplaybreaks

\newtheorem{conjecture}[theorem]{Conjecture}

\theoremstyle{definition}
\newtheorem{definition}[theorem]{Definition}
\newtheorem{remark}[theorem]{Remark}

{{\sc Proof of Theorem~\ref{earthquakethm}}}
{{\sc q.e.d.} \\}

\newenvironment{thm:1relvw}
{{\sc Proof of Theorem~\ref{thm:1relvw}.}}
{{\sc q.e.d.} \\}
\newenvironment{genNTheorem 2}
{{\sc Proof of Theorem~\ref{genNTheorem 2}.}}
{{\sc q.e.d.} \\}
{{\sc Proof of Claim~\ref{convvq}.}}%
{{\sc q.e.d.} \\}
{{\sc Proof of Lemma~\ref{cd2general}.}}%
{{\sc q.e.d.} \\}
{{\sc Proof of Lemma~\ref{usordv}.}}%
{{\sc q.e.d.} \\}
  \newenvironment{support}%
{{\sc Proof of Theorem~\ref{thm:supptmeasure}.}}%
{{\sc q.e.d.} \\}
{{\sc Proof of Theorem~\ref{mainthmpushforw}.}}%
{{\sc q.e.d.} \\}
 \newenvironment{Length=totvar}%
{{\sc Proof of Theorem~\ref{Length=totvar}.}}%
{{\sc q.e.d.} \\}
{{\sc Proof of Theorem~\ref{convthmrepresentearth}.}}%
{{\sc q.e.d.} \\}
  \newenvironment{Length-totvarThur}%
{{\sc Proof of Theorem~\ref{Length-totvarThur}.}}%
{{\sc q.e.d.} \\}

  \newenvironment{thm:existence}%
{{\sc Proof of Theorem~\ref{thm:existence}.}}%
{{\sc q.e.d.} \\}
\numberwithin{equation}{section}

{{\sc Proof of Theorem~\ref{regularity}.}}%
{{\qed} \\}
\newenvironment{proof:main}%
{{\sc Proof of Theorem~\ref{ex}.}}%
{{\qed} \\}

\newenvironment{proof:pluriharmonic}%
    {{\sc Proof of Theorem~\ref{theorem:pluriharmonic}.}}%
  {{\qed} \\}  
  
  \newenvironment{proofof(iv)}%
    {{\sc Proof of $(iv)$.}}%
  {{\qed} \\}  
\newcommand{\norm}[1]{\left\Vert^\sharp 1\right\Vert}
\newcommand{\abs}[1]{\left\vert^\sharp 1\right\vert}
\newcommand{\R}{\mathbb R}

\newcommand{\HH}{\mathbb H}

\begin{document}
\title{Analytic properties of  Stretch maps and geodesic laminations}
\author[Daskalopoulos]{Georgios Daskalopoulos}
\address{Brown Univeristy \\
Providence, RI}%02912}
%\email{daskal@math.brown.edu}
\author[Uhlenbeck]{Karen Uhlenbeck}
\address{University of Texas Austin, TX and Institute for Advanced Study, Princeton, NJ}
%\email{uhlen@math.utexas.edu }
\thanks{GD supported in part by NSF DMS-2404915}
\maketitle

\begin{abstract}
In a 1998 preprint (cf. \cite{thurston}), Bill Thurston outlined a Teichm\"uller theory for hyperbolic surfaces based on maps between surfaces which minimize the Lipschitz constant (minimum stretch or best Lipschitz maps). In this paper we continue the analytic investigation which we began in \cite{daskal-uhlen1}.  In the spirit of the construction of infinity-harmonic functions, we produce best Lipschitz maps $u$ as limits $p \rightarrow \infty$ of minimizers of  $p$-Schatten integrals ($p$-Schatten harmonic maps) in a fixed homotopy class  between  hyperbolic surfaces. We address existence and regularity of $p$-Schatten harmonic maps with the latter, due to higher degeneracies,  being significantly harder  than for ordinary $p$-harmonic maps. Moreover, we  construct  Lie algebra valued dual functions which minimize a dual ($1/p+1/q=1$) $q$-Schatten  integral  and limit as $q \rightarrow 1$ to a locally defined, Lie algebra valued function $v$ of bounded variation.  One of  the main results of the paper is the surprising fact that the support of the measure  $dv$ (the derivative of $v$) lies on the canonical geodesic lamination constructed by Thurston \cite{thurston} and further studied by Gueritaud-Kassel \cite{kassel}. In the sequel paper \cite{daskal-uhlen2} we will show how these Lie algebra valued measures induce  a transverse measure on the canonical lamination and relate to other aspects of Thurston theory.%This is the second in a series of papers where we attempt to provide an analytic understanding of Thurston's best Lipschitz maps between hyperbolic surfaces  and the asymmetric metric on Teichm\"uller space. It is a continuation of our previous work on infinity-harmonic functions and least gradient maps \cite{daskal-uhlen1}. In the current paper we introduce a variation of $p$-harmonic maps between surfaces that minimize the Schatten $p$-norm (instead of the $L^p$-norm) of the gradient and study their regularity properties. We proceed to show that the modified $p$-harmonic maps converge as $p \rightarrow \infty$ to an infinity-harmonic map which is also a best Lipschitz map in the same homotopy class. The maximum stretch set of the infinity-harmonic map contains a canonical geodesic lamination, which can also be identified with Thurston's chain recurrent lamination $\mu(g,h)$ (cf. \cite[Theorem 8.2]{thurston}) associated to the hyperbolic structures and the homotopy class. 
%Minimizing the dual to the Schatten $p$-functional defines a closed 1-form with values on an associated flat bundles. These converge as $q \rightarrow 1$ to closed, mass minimizing  1-currents with values on the flat bundle. One of our main results  is to show that  the currents have support   on the canonical lamination. 
%We further study the question whether the limiting harmonic map is optimal in the sense that its stretch set is exactly equal to the canonical lamination. We show this is in fact true if the lamination consists of closed geodesics and conjecture to be true in  general. We postpone  the study of the transverse measures associated to the 1-currents, as well other topological questions for the sequel paper. 
\end{abstract}

\section{Introduction}

In a 1998 preprint Thurston proposed a model for Teichm\"uller space based on best Lipschitz maps between Riemann surfaces. Using topological methods he constructs best Lipschitz maps $u$ with maximal stretch locus  on a geodesic lamination, which is mapped linearly by $u$ to the corresponding lamination in the target.   Thurston conjectures  there may be a simpler approach based on the duality between measures and $L^\infty$ norms but gives no clue how to do it.  This paper  carries out the analysis needed to develop  such a theory. We leave the connection with topology  to the subsequent paper \cite{daskal-uhlen2}.

There are two goals to the paper.  First, we initiate  a theory of best Lipschitz maps between surfaces and their approximations (Schatten harmonic maps) in the spirit of infinity-harmonic functions.
There was not much known about this problem in the literature (other than \cite{smart}) so we  had to develop most of the analysis from scratch. Along the way, we proved some straight PDE type results which give some insight into the behavior of the solutions to these somewhat unusual equations. However, the real surprise was the appearance of dual transverse measures with values in the Lie algebra and with support on Thurston's canonical lamination.  Thurston only hints of their possibility and gives  no clue as to how to define them, find them or use them.  Infinity harmonic maps to manifolds of dimension greater than one are considerably harder to understand analytically than infinity-harmonic functions.  That there are enough estimates allowing us to prove that the support of the dual measures is on the canonical lamination is at the limits of our understanding of these best Lipschitz maps. 

This introduction contains an outline of the topics and a statement of the main results.  We end  with a description of the more topological results in \cite{daskal-uhlen2}. Section~\ref{sect1} contains the construction of the infinity-harmonic map and their $p$-approximations. For this,
we fix a homotopy class of maps given by $f:M \rightarrow N$. If the boundary $\partial M \neq \emptyset$, either fix the boundary data (Dirichlet problem) or put no restriction on the boundary (Neumann problem). {\it If the dimension of  $N$ is greater than one,  the study of the usual $p$-harmonic maps as $p \rightarrow \infty$ produces a map which minimizes $\max |df|$, which is not the Lipschitz constant} (cf. \cite{katz}). The geometric significance of this equation is unknown.   Sheffield-Smart (cf.  \cite{smart}) suggest using the approximation of $\int_Ms(df)^p*1$, for $s(df$) the largest singular value of $df$.  For finite $p$ this unfortunately does not lead to a recognizable elliptic partial differential equation as its Euler-Lagrange equation. We instead use a Schatten norm, in which the integrand is essentially the sum of the $p$-th powers of the eigenvalues of $df$.  Let
\[
J_p(f) = \int_M TrQ(df)^p *1 
\]
where 
$Q(df)^2 =  dfdf^T$ is a non-negative symmetric linear map mapping the tangent space $T_f N$ to itself.  The Euler-Lagrange equations  of $J_p$ are 
\begin{equation}\label{intrel}
D^*Q(df)^{p-2}df=0
\end{equation}  
where $D=D_f$ is the pullback of the Levi-Civita connection on $f^{-1}(TN)$.
We prove existence (cf. Theorem~\ref{thm:pdirichlet}) and uniqueness (cf. Corollary~\ref{cor:unique2}) for solutions of (\ref{intrel}).
We end the section with a variational construction of  the infinity-harmonic map.
Here we show that as $p \rightarrow \infty$, the minimizers  of $J_p$ converge to  a best Lipschitz map. This is Theorem~\ref{thm:existence}.
\

In Section~\ref{sec3} we construct the dual functions. The dual functions arise from  conservation laws associated with the symmetries of the target. Technically they have values in the dual of the Lie algebra, but since we constantly use the geometry coming from the indefinite invariant inner product, we will refer to them as Lie algebra valued. 
In \cite{daskal-uhlen1}, we found the dual function by inspection, but we found its Lie algebra valued counterpart in the present paper only by looking at the conserved quantities arising from the symmetries of the target via Noether's theorem. We describe the flat bundle structure needed to encode the local action of $SO(2,1)$ on $N$. This is $E=\tilde M \times_\rho \R^{2,1}$ where $\rho$ has image in the local isometries of  $N$.  The dual functions are only defined locally. We will study their global formulation in the sequel paper \cite{daskal-uhlen2}. The description of the closed 1-form needed to obtain them is Theorem~\ref{Prop:Elag-bund} listed below:

\begin{theorem}  Let $u=u_p$ satisfy the $J_p$-Euler-Lagrange equations.  Then, in the distribution sense
$d* (S_{p-1}(du) \times u)=0.$ 
\end{theorem}
Here $S_{p-1}(du)=Q(du)^{p-2}du$ and  $d$ is the derivative computed using the flat on $ad(E)$.  If we set $ V_q = *(S_{p-1}(du) \times u)$ then $V_q$ is the closed $1 = \dim M-1$  form predicted by Noether's theorem. We can set locally $dv_q = V_q$.
As in the case of functions, these dual fields also satisfy the Euler-Lagrange equations for a functional based on Schatten $q$ norms, $1/q + 1/p = 1.$ This is the content of Theorems~\ref{TTheorem A6} and~\ref{TTheorem A66}.
%Ignoring regularity issues of $u$ for the moment, we can explain these theorems as follows: Given a solution $u$ satisfying the $J_p$-Euler-Lagrange equations, we can view $*S_{p-1}(du)$ as a section of the pullback bundle $u^{-1}(TN) \otimes T^*M$. Then the minimum of the functional $J_q(\xi)=||*S_{p-1}(du)+D_u\xi||^q_{sv^q}=||*S_{p-1}(du)+(d\xi)_u||^q_{sv^q}$ over all sections $\xi$ of  $u^{-1}(TN)$ is attained at $\xi=0$. Similarly for $J_q(\psi)=||V+(d\psi)_u||^q_{sv^q}$ where $V=*S_{p-1}(du) \times u$.  Here $\psi$ a section of the subbundle of the Lie algebra bundle $\mathfrak p_u$ corresponding to the  positive definite part of the Cartan decomposition defined by $u$ (cf. Proposition~\ref{helpemb}). In both cases subscript $u$ means projection onto the positive part of the bundle.

In our first paper \cite{daskal-uhlen1}, we omitted to examine the conservation laws which arise from the local symmetries on the domain  surface $M.$  Section~\ref{consnoetdom} remedies this situation. We   start by computing the energy momentum tensor  $T=T(du)$ associated to  a $J_p$-minimizer $u=u_p$  and show that  it is divergence free. We prove the theorem in somewhat greater generality covering a wider class of variational problems between Riemannian manifolds (another example is ordinary $p$-harmonic maps). More specifically, let $I_p$ be any functional of the type considered in the beginning of Section~\ref{enmomtens}. Then:

\begin{theorem}\label{genNTheorem 21} If $u=u_p$ is a minimum in $W^{1,p}(M,N)$ of $I_p$, then $D^*T= 0$, i.e the symmetric (0,2)-tensor $T=T(du)$ is divergence free with respect to the covariant derivative.
\end{theorem}

This is Theorem~\ref{genNTheorem 2}.
In the case of  hyperbolic surfaces, we push $*T$ forward to the Lie algebra bundle to obtain, as predicted by Noether's theorem, a {\it closed} $(n-1)=1$-form $W$ with values in the Lie algebra.   
Note that this construction depends only on the local symmetric space structure of the domain.  We conjecture that some version is true for a variety of integrands and domains $M$ with $\dim M \geq 2.$

In Section~\ref{sect:regtheor} we discuss results on regularity of minimizers of the functionsl $J_p$.
This section deals with two dimensions only. The functional $J_p$ leads to an Euler-Lagrange equation which is elliptic as long as the eigenvalues of the derivative of the map are non-zero.  In dimension 2, at least $C^{1,\alpha}$ regularity would ordinarily  be expected.  However, we were unable to prove this even for $p = 4.$ Possible disparity between the two eigenvalues (which is why we picked this integrand) invalidates standard techniques such as hole filling. Since so much of the later chapters involve computations on solutions of the Euler-Lagrange equations, we include the regularity theorem that is sufficient for our applications.  Note that the apriori estimates on smooth solutions are quite easy to obtain.  However, to use these estimates, we would have to expand our integrands to a one parameter family that is known to have smooth solutions in the family up until the final point we are seeking to estimate.  This actually could be done, but is not clearly easier and is certainly not more applicable.  Our main regularity result is 
Theorem~\ref{mainregtH} and its Corollary~\ref{cormainredsh}. 
Here  we abbreviate $Q(du)^{p/2-1}du  = S_{p/2}(du)$.

\begin{theorem}If $u=u_p$ satisfies the $J_p$-Euler-Lagrange equations in $\Omega \subset M$, and $\Omega' \subset \Omega$, then $S_{p/2}(du) \big|_{\Omega'} \in H^1(\Omega')$ and 
\[
          ||S_{p/2}(du)||_{H^1(\Omega')} \leq  kpJ_p(du\big|_{\Omega'})^{1/2}.
  \]
Here $k$ depends on the geometry of $\Omega' \subset \Omega \subset\HH$ but not on $p.$   Moreover $du|\Omega'$ is in 
$L^s$ for all $s.$
\end{theorem}

For the rest of the paper following this theorem, we restrict ourselves to maps between hyperbolic surfaces. We also make the standing assumption that the best Lipschitz constant is at least one.
In order to follow the rest of the paper, it is necessary to absorb the structure of the canonical laminations $\lambda$ in $M$ and 
$\lambda^\land$ in $N$ determined by a homotopy class of maps. We refer the reader who did not find the description in the beginning of the introduction satisfactory to Definition~\ref{canlamin}, and the papers \cite{thurston} and  \cite{kassel}. In fact, the above papers papers contain a proof of the existence of these canonical laminations, and show that every best Lipschitz map must contain $\lambda$ in the set of points at which the best Lipschitz constant is taken on (=maximum stretch set).  Hence our infinity-harmonic maps contain $\lambda$ in their stretch set. Gueritaud-Kassel use the term {\it optimal best Lipschitz} for best Lipschitz maps whose maximum stretch set is $\lambda.$ Note that Thurston's stretch maps are not optimal, but Gueritaud-Kassel prove the existence of optimal best Lipschitz maps.

In Section~\ref{lipqgoesto1} we consider the limit $q \rightarrow1$. Here $q$ is the conjugate to $p$ satisfying $1/p+1/q=1$.
As in \cite{daskal-uhlen1}, we normalize $S_{p-1}$ and obtain a measure in the limit as $p$ goes to infinity. In this paper, we similarly consider the rescaled tensor
\[
S_{p-1}= S_{p-1}(\kappa_p du_p) = \kappa_p^{p-1}Q(du_p)^{p-2}du_p
\]
 for a normalizing factor $\kappa_p$, and $|S_p| = \kappa_p(S_p ;du_p)^\sharp=TrQ(du_p)^p$. The next theorem is a combination of Theorem~\ref{TTheorem A5} and Theorem~\ref{thm:limmeasures0}.

 \begin{theorem}\label{thm:limmeasures1694} Given a sequence $p \rightarrow \infty$, there exists a subsequence (denoted again by $p$), a real-valued positive Radon measure $|S|$, a Radon measure $S$ with values in $T^*M \otimes E$ and a  Radon measure $ V$ with values in $T^*M \otimes ad(E)$ such that
\begin{itemize}
\item $(i)$ $ |S_{p-1}| \rightharpoonup |S| $ and $\int_M |S|*1 = 1$
\item $(ii)$ $ S_{p-1} \rightharpoonup S$ and $V_q \rightharpoonup  V$    
\item $(iii)$ The total masses of $S$ and $V$ are  one and two respectively
\item $(iv)$ $ S=S_u$ and $V=V_u$ 
\item $(v)$The supports of $S$, $V$ are equal and contained in the support of $|S|$ 
\item $(vi)$ $V$ is closed with respect to the flat connection  on $ad(E)$.
\end{itemize}  
\end{theorem}

In the above mass means total variation adapted to our situation. For further details we refer to Definitions~\ref{measrcurr} and~\ref{measrcurr2}.

A similar limiting construction can be done for the Noether currents coming from the domain symmetries. Let $F$ be the flat bundle $F=\tilde M \times_\sigma \R^{2,1}$, where $\sigma: \pi_1(M) \rightarrow SO^+(2,1)$ defines the hyperbolic structure on $M$. Then, there exist Radon measures $T$ and $W$ with values respectively in $T^*(M) \otimes F$ and $T^*(M) \otimes ad (F)$ which are the weak limits of the (appropriately rescaled)
tensors $T_q = (S_{p-1}(du_p), du_p)^\sharp $ and $W_q= T_q \times id$:   

\begin{theorem}\label{thm:limmeasures4501}  Given a sequence $p \rightarrow \infty$ ($q \rightarrow 1$), there exists a subsequence (denoted again by $\{p\}$) 
%a real-valued positive Radon measure $|S|$  a closed 1-current $*S$ with values in $E$ 
and  Radon measures $T$ and $W$ with values in $Sym^2(T^*M)$, $T^*M \otimes ad(F)$ respectively such that after normalizing as above:
\begin{itemize}
\item $(i)$  $ T_q \rightharpoonup T$, $W_q=T_q \times id \rightharpoonup W$ 
\item $(ii)$ $dW=0$  with respect to the flat connection on $ad(F)$ and  $W_{id}=W$
\item $(iii)$ $Tr_gT=|S|$ and $*(\omega_{mc}, W)^\sharp=2|S|$
\item $(iv)$The supports of  $T$, $W$ and $|S|$ are equal.
\item $(v)$ The total masses of $T$ and $W$ are  one and two respectively. 
\end{itemize}
%\begin{itemize}
%%\item $(i)$ $ |S_{p-1}| \rightharpoonup |S| $ where  $\int_M |S|*1 = 1$
%\item $(ii)$    $ V_p \rightharpoonup V$
%\item $(iii)$ the measure $S$ and the currents $|S|$ and $V$ have the same support
%\item $(iv)$ The closed current $*S$ and $V$ are mass minimizing with respect to the best Lipschitz map $u$.
%\end{itemize}  
%\end{theorem}
%Moreover,
%\begin{theorem}\label{thm:supptmeasure} The support of the currents $S$ and $V$  is contained in the canonical geodesic lamination $ \lambda$ associated to the hyperbolic metrics $g$, $h$ and the homotopy class.
\end{theorem}
Primitives $w$ can be found for $W,$ just as for $V$. Here $w$ also represents a local Lie algebra valued function of bounded variation. We will investigate the global properties of $v$ and $w$ in connection with transverse measures in the next paper \cite{daskal-uhlen2}.

In Section~\ref{sec7} we relate the  support of the measures $|S|$, $S$, $V$, $T$ and $W$ to the canonical lamination.
Recall our remarks about the canonical lamination $\lambda$ and its image $\lambda^\land.$ The existence of the measures  is in itself not interesting unless we know some useful geometric properties of them. In this section, we show that their supports are on the geodesic lamination $\lambda.$

 The next theorem is one of the main results of the paper. It is  proved by comparison with  optimal best Lipschitz maps and  says that the supports of the measures are on the canonical $\lambda$, not just in the maximum stretch set of $u.$ 
 
 \begin{theorem} 
The support of the measure $|S|$  is contained in the canonical  lamination  $\lambda$ associated to the hyperbolic metrics on $M$, $N$ and the homotopy class.
\end{theorem} 
Together with Theorems~\ref{thm:limmeasures1694} and~\ref{thm:limmeasures4501}, it follows  that the supports of  $S$,  $V$, $T$ and $W$ are contained in the canonical  lamination.
  
 Here is a list of the topics which will be addressed in the sequel paper \cite{daskal-uhlen2}:\\
First,  is the global description on the measures $V$, $W$ and their primitives. The homology classes of  $V$ and $W$ are elements of $H^1(M, ad(E))$ and $H^1(M, ad(F))$ respectively. The local primitives $v$, $w$ satisfying $V = dv$, $W= dw$ form global  sections of affine bundles with
linear structure $ad(E)$ and $ad(F)$, have bounded variation and, as a consequence of the support theorem, are constant on the plaques of the canonical lamination. In other words, $V$ and $W$ can be realized as transverse measures with values in a flat Lie algebra bundle and their primitives as  transverse cocycles (cf. \cite{bonahon}). However,  unlike Bonahon's transverse cocycles,  $v$ and $w$ are not invariant under deck transformations. They satisfy globally a twisted affine equivariance condition (composition of the adjoint representation and a translation in the Lie algebra). 

The 1-currents  $W$, $V$ can be described in terms of the geodesic flow on the lamination $\lambda$ and an induced transverse measure $\mu$  on $\lambda$ \footnote{in the topology literature (cf. \cite[Chapter 8]{thurston2}) it is customary to require that transverse measures are of full support; in our convention a transverse measure could be zero on  part of the lamination.}.    More precisely,  there exists a transverse measure $\mu$ on $\lambda$ such that $W=Bd\mu$ and $V=B^\land(u) d\mu$ where $B$, $B^\land$ denotes the geodesic flow of $\lambda$, $\lambda^\land=u(\lambda)$ respectively.
Moreover,  the masses (total variation) of $W=dw$ and $V=dv$ computed using the hyperbolic metric are proportional to the length of the laminations. 

Note that $H^1(M, ad(F))$ and $H^1(M, ad(E))$ can be identified with the tangent spaces to the character variety, or equivalently the Teichm\"uller space at $M$, $N$.  Using the integral trace pairing (symplectic form), $W$ corresponds to the generator of the earthquake flow along the lamination $(\lambda, \mu).$  One way we see this is  by showing that it is dual to the derivative of the length functional of $\lambda$ with respect to the variation of the hyperbolic structures on $M.$  But we can also show this directly by constructing the  earthquake map. In other words, the vector field induced on $M$ by $W = dw$ with the equation  $\zeta (x) = w x$ is a Killing field away from $\lambda$ with a jump across this lamination.  This is an exact description of an earthquake map (cf. \cite{thurston3}).

 A similar picture holds for $V$, where $v$ with $dv = V$ describes the earthquake flow  along the image lamination $\lambda^\land=u(\lambda)$. The geometry of $M$ is identified with that of $N$ in a neighborhood of the laminations, the push-forward of the vector field $\zeta$ describes  a Killing vector field on the image $N$ with jumps across the lamination $u(\lambda)$. 
 
There are a number of  interesting open problems that we have not addressed in this paper: 

To begin, recall that in the case of maps to $S^1$ the local primitive $v$ of the closed 1-current $V$ is locally a function of least gradient (cf. \cite{daskal-uhlen1}). The level sets of a least gradient function are geodesics and this gives another way to recover the canonical lamination (cf. \cite{aidan}). It turns out that the map $v$ here also satisfies  a least gradient property, although it is  harder  to state.  As there is not a theory of least gradient vector valued functions readily available in the literature, we postpone this discussion to a future paper.

Another interesting problem is that of uniqueness and in most cases we cannot prove uniqueness.  Also, we  cannot show that our infinity-harmonic maps are optimal. We can  only give affirmative answers in the case that the laminations consist of closed geodesics. This will also be discussed elsewhere.

Because of the connection with Thurston theory, this paper treats only $G = SO(2,1)$.  The results in the beginning of the paper hold for all image manifolds $N$ (at least of non-positive curvature), but the later sections are written explicitly for $SO(2,1).$  There is considerable interest in allowing more general targets.  Our results go over with very little change to the case of $SO^+(n,1)/SO(n)$ and the Lipschitz constant $L > 1$.  For $L < 1$, the analysis carries over entirely. However, we do not know what replaces the canonical lamination $\lambda$; it is quite possible the analysis of $dv$ and $dw$, which are well-defined, will be helpful in determining this. 
 For $SL(n,R)/SO(n)$, one can, of course, define infinity-harmonic maps. Noether's theorem still applies, but the support theorem is more complicated to execute.  Again, the main obstacle is  the lack of a replacement for the canonical laminations.

{\bf Acknowledgements.}
We would like to thank Camillo De Lellis and  Athanase Papadopoulos for useful discussions during the preparation of this manuscript. We would also like to thank Krzysztof Ciosmak for suggesting  a possible connection between this paper and his work. \cite{cios1}, \cite{cios2}.

\section{Best Lipschitz maps and their $p$-approximations}\label{sect1}
In this section we introduce a new version of $p$-harmonic maps between Riemannian manifolds. They are  defined as  critical points of a functional $J_p$ given by the integral of the $p$-power of the Schatten norm of the  gradient. We first review some simple facts from linear algebra. These include the definition of the Schatten norms on the space of matrices and their basic properties. We next  study basic properties of the functional $J_p$, like convexity, in the case when the target manifold has non-positive curvature. In this case, we conclude existence and uniqueness of solutions the same way as for  harmonic maps. We postpone the rather difficult question about regularity  until Section~\ref{sect:regtheor}.
We conclude the section with the construction of a special type of best Lipschitz maps which we call infinity-harmonic. These maps are limits of minimizers of  $J_p$-functionals as $p \rightarrow \infty$.

\subsection{Preliminaries} 
Given $V_1$ and $V_2$  positive definite inner product  spaces and $A \in Hom(V_1, V_2)$ denote $A^T \in Hom(V_2, V_1)$ the adjoint. Here the inner products are used to identify the spaces with their duals. First note that for $ B \in Hom(V_1, V_2)$,
\begin{equation}\label{eqn:traceinn}
Tr(A^TB)=Tr(AB^T)=(A, B)
\end{equation}
is nothing but the trace inner product of the matrices $A$ and $B$. Denote by $|A|_2=(A, A)^{1/2}.$
Set
\[
Q(A)=(AA^T)^{1/2} \ \  \mbox{and} \ \  \mathcal Q(A)=(A^TA)^{1/2}. 
\]
Define by $s_1(A) \geq s_2(A) \geq ...\geq s_r(A) \geq 0$ 
their common eigenvalues  (singular values of $A$) where $r=\min\{\dim V_1, \dim V_2\}$.
%\subsection{Preliminaries}
%We start with some simple linear algebra lemmas.
%\begin{lemma}\label{traceinn} Given $V$ and $W$  inner product spaces and $A \in Hom(V, W)$ we denote $A^T \in Hom(W, V)$ the adjoint, where we use the inner products to identify the spaces and their duals. For $ B \in Hom(V, W)$,
%\begin{equation}\label{eqn:traceinn}
%Tr(A^TB)=Tr(B^TA)=<A, B>
%\end{equation}
%\end{lemma}

%\begin{lemma}\label{equalaat} Let $A \in Hom(V, W)$. Then
%\begin{equation}\label{Poiu}
%Tr(AA^T)^m=Tr(A^TA)^m.
%\end{equation}
%%If in addition $dim V=dim W$, then 
%%\begin{equation}\label{Poiu}
%%Tr(AA^T+c^2I)^m=Tr(A^TA+c^2I)^m.
%%\end{equation}
%%where $I$ is the identity matrix in $W$ or $V$. 
%\end{lemma}
%\begin{proof}Assume the dimensions of $V$ and $W$ are $n$ and $s$ respectively.
%By the singular value decomposition theorem, choose $U_1$ unitary in $V$ and $ U_2$ unitary in $W$ such that 
%\[
%A=U_1DU_2^T
%\]
%where $D$ is a diagonal $s \times n$ matrix.  Thus
%\[
%AA^T=U_1DD^TU_1^T, \ \ \ A^TA=U_2D^TDU_2^T.
%\]
%Since $U_i^T=U_i^{-1}$ and $D$ is diagonal,
%\[
%tr(AA^T)^m=tr(DD^T)^m=tr(D^TD)^m=tr(A^TA)^m
%\]
%proving (\ref{Poiu}).
%\end{proof}
%

\begin{definition}\label{schatten} Let $1 \leq p < \infty$ and $A \in Hom(V_1, V_2)$. Define the $p$-Schatten norm
\[
|A|_{sv^p}= \left( TrQ(A)^p \right)^{1/p}= \left( Tr\mathcal Q(A)^p \right)^{1/p}=\left( \sum_{i=1}^r s_i(A)^p \right)^{1/p}.
\] 
Extend the definition at $p=\infty$ by setting 
\[
|A|_{sv^\infty}=\sup_{|a|=1}|A(a)|
\]
the operator norm of $A$. Equivalently,
$|A|_{sv^\infty}= s_1=s_1(A)$
  is the largest singular value of $A$. For convenience we will denote $s_1(A)$ simply by $s(A)$.
 \end{definition}
 
 The next proposition  lists some fairly standard properties of Schatten norms found in the literature (see for example \cite{bhatia}).   
\begin{proposition}\label{shat1}
\begin{itemize}
\item $(i)$ For $1/p+1/q=1/r, \ p,q,r \in (1, \infty)$ 
\[
|AB^T|_{sv^r} \leq |A|_{sv^p}|B|_{sv^q} 
\]
\[
|Tr(AB^T)| \leq |A|_{sv^p}|B|_{sv^q}
\]
\item $(ii)$ For $1  \leq p \leq q \leq \infty$,
\[
|A|_{sv^1} \geq |A|_{sv^p}  \geq |A|_{sv^q}  \geq |A|_{sv^\infty}
\]
\[
\lim_{p \rightarrow \infty}|A|_{sv^p}=|A|_{sv^\infty}.
\]
\end{itemize}
\end{proposition}

It is worth noting that property $(ii)$, though elementary, is the starting point  of this paper. 

%\begin{proof} For $(i)$ see \cite[(IV. 42), (IV. 43)]{bhatia}. For $(ii)$ combine $(i)$ for $r=1$ with \cite[(II. 40)]{bhatia}. Properties
%$(iii)$ and  $(iv)$ follow immediately from the following well known facts:  for a fixed $r$-tuple of non-negative numbers $(x_1,...x_r) \in \R^r$ the function
%\[
%p \mapsto \left(\sum_{i=1}^r x_i^p \right)^{1/p}
%\]
%is non-increasing in $p$ and as $p\rightarrow \infty$ converges to $\max_{i=1,...r}x_i$.
%\end{proof}

In this paper we will use the induced norms  on spaces of sections of vector bundles. More precisely, let $V_1, V_2$ be Riemannian vector bundles over a Riemannian manifold $(M,g)$ and $A: M \rightarrow Hom(V_1,V_2)$ a section. Define the $p$-Schatten norm of $A$
\[
||A||_{sv^p}=\left(\int_M |A|^p_{sv^p}*1\right)^{1/p}\ \ \mbox{for} \ \ 1\leq p < \infty
\]
and
\[
||A||_{sv^\infty}= \sup_M |A|_{sv^\infty}.
\]

The proofs of the following  Proposition and Lemma are elementary:

\begin{proposition}\label{ineqscoten} The norm $||.||_{sv^p}$  is equivalent to the $L^p$ norm. More precisely,
\begin{equation}\label{lemma:normineq0}
\frac{1}{\sqrt r} ||A||_{L^p} \leq ||A||_{sv^p} \leq  ||A||_{L^p}
\end{equation}
where $r=\min \{\dim V_1, \dim V_2 \}$. Furthermore, for sections $A$, $B$ and $1/p+1/q=1$,  $p, q \in (1, \infty)$
\begin{eqnarray}\label{holds}
||AB^T||_{sv^1} \leq ||A||_{sv^p}||B||_{sv^q} \ \mbox{and} \ \left| \int_M Tr(AB^T) \right| \leq ||A||_{sv^p}||B||_{sv^q}.
\end{eqnarray}
\end{proposition}
%\begin{proof}To show that $||.||_{sv^p}$ is a norm on the space of sections, it suffices to check subadditivity. This  follows from a standard argument:
%\begin{eqnarray*}
%||A+B||^p_{sv^p}&=&\int_M|A+B|_{sv^p}^p*1\\
%&\leq&\int_M(|A|_{sv^p}+|B|_{sv^p})|A+B|_{sv^p}^{p-1}*1\\
%&\leq&\left(\left(\int_M |A|_{sv^p}^p*1 \right)^{1/p}+\left(\int_M |B|_{sv^p}^p*1 \right)^{1/p}\right)\left(\int_M |A+B|_{sv^p}^p\right)^{1-1/p}\\
%&=&(||A||_{sv^p}+||B||_{sv^p})||A+B||^{p-1}_{sv^p}.
%\end{eqnarray*}
%For the second statement,
%\[
%|A|_{sv^p}^p= \sum_{i=1}^r s_i(A)^p\leq   \left(\sum_{i=1}^r s_i(A)^2 \right)^{p/2}=|A|_2^p
%\]
%and
%\[
%|A|_{sv^p}^p= \sum_{i=1}^r s_i(A)^p\geq   s_1(A)^p = (s_1(A)^2)^{p/2} \geq \left(\frac{1}{r}\sum_{i=1}^r s_i(A)^2 \right)^{p/2}
%= \frac{1}{r^{p/2}}  |A|_2^p.
%\]
%Finally, from Proposition~\ref{shat1}(i),
%\begin{eqnarray*}
%||AB^T||_{sv^1} &=& \int_M |AB^T|_{sv^1} 
%\leq \int_M |A|_{sv^p} |B|_{sv^q} \\
%&\leq&  \left(\int_M |A|_{sv^p}^p\right)^{1/p}  \left(\int_M |B|_{sv^q}^q\right)^{1/q}
%=||A||_{sv^p}||B_{sv^q}.
%\end{eqnarray*}
%The other inequality in (\ref{holds}) also follows  in a similar way.
%\end{proof}

\begin{lemma}\label{lemma:normlim} Let $V_1, V_2$ be Riemannian vector bundles of dimension $r$ over a Riemannian manifold $(M,g)$ and $A: M \rightarrow Hom(V_1,V_2)$ a section. Then
\begin{itemize}
\item $(i)$
\[
\lim_{p \rightarrow \infty} ||A||_{sv^p}= ||A||_{sv^\infty}=s(A).
\]
\item $(ii)$ If $1 \leq s<p<\infty$, then
\[
\frac{1}{(r \ vol(M))^{1/s}}||A||_{sv^s}\leq \frac{1}{vol(M)^{1/p}} ||A||_{sv^p}.
\]
\end{itemize}
\end{lemma}
%\begin{proof}Let $s_1(x) \geq s_2(x)\geq ... \geq s_r(x)$ be the eigenvalues of $Q(A(x))$ pointwise. Then,
%\begin{eqnarray*}
%||A||_{sv^p}&=&\left(\int_M \sum_{i=1}^r s_i(x)^p *1\right)^{1/p}
%\leq ||s_1||_{L^\infty}\left(\int_M r dx \right)^{1/p}.
%\end{eqnarray*}
%Thus
%\[
%\limsup_{p \rightarrow \infty} ||A||_{sv^p} \leq ||A||_{sv^\infty}.
%\]
%From Proposition~\ref{shat1}$(ii)$, 
%\[
%||A||_{sv^p}=\left(\int_M |A|_{sv^p}^p \right)^{1/p} \geq \left(\int_M |A|_{sv^\infty}^p\right)^{1/p}.
%\]
%By taking limits, 
%\[
%\liminf_{p \rightarrow \infty} ||A||_{sv^p} \geq  ||A||_{sv^\infty}.
%\]
%This proves $(i)$. For $(ii)$,
%\begin{eqnarray*}
%||A||_{sv^s} &= &\left(\int_M \sum_{i=1}^r s_i(x)^s *1\right)^{1/s}
%\leq \left(\int_M r s_1(x)^s *1\right)^{1/s}\\
%&= & (r vol(M))^{1/s}\left(\avint_{\ \ M}  s_1(x)^s *1\right)^{1/s}\\
%&\leq &(r vol(M))^{1/s} \left(\avint_{\ \ M}  s_1(x)^p *1\right)^{1/p} \\ 
%&\leq & \frac{(r vol(M))^{1/s}}{vol(M)^{1/p}} ||A||_{sv^p}.
%\end{eqnarray*}
%\end{proof}

The next Lemma will be used in  the proof of Corollary~\ref{compprojnorm}. It will be important to estimate tangent vectors to the hyperbolic space $\HH$ at nearby points.

\begin{lemma}\label{ineqsvp} Let $A, B : V \to W $  be linear transformations between finite dimensional vector spaces with inner product. Suppose that for all $v \in V$, $|B(v)| \leq |A(v)|$. Then,   their singular values satisfy the inequality $s_k(B) \leq s_k(A)$ for all $k$.  
In particular, for  all $1 \leq p< \infty$, $|B|_{svp} \leq |A|_{svp}$, and $s(B) \leq s(A)$.  
\end{lemma}
\begin{proof}
The assumption $|B(v)| \leq |A(v)|$ for all $v$ implies that $BB^T \leq AA^T$ as positive definite matrices. The lemma follows from the  Courant-Fischer-Weyl minimax principle 
(cf \cite[Corollary III.1.2]{bhatia}).
\end{proof}

The following  elementary consequence of convexity  could not be found in the references. Therefore, we include a proof.
\begin{lemma}\label{lemmaconvvnorms} For $A, B \in Hom(V_1, V_2)$ and $2<p<\infty$
\[
    p(Q(A)^{p-2}A, B - A)  \leq |B|_{sv^p}^p-|A|_{sv^p}^p 
   \]
\end{lemma}
\begin{proof} The function $j(A) = Tr Q(A)^p$ is convex, and differentiable. Furthermore,
\[   
   (dj)_A(C)=  p(Q(A)^{p-2}A, C).
\]
If $j$ is a convex function on a vector space, it is always true that
\[
   (dj)_A(C) \leq j(A + C) - j(A).
\]
The proposition follows by setting $C = B - A$.
\end{proof}

\subsection{The functional}\label{sect:funtional}
Let $(M, g)$  be a compact Riemannian manifold with  boundary $ \partial M$ (possibly empty) and $\dim M=n$.  Let $(N, h)$ be closed.  Throughout the paper we will denote the inner product coming from the domain metric $(.;.)$ and the one from the target metric $(.,.)^\sharp.$ The notation $(.;.)^\sharp$ means we use both metrics.

For   $1< p < \infty$ consider the subspace of  maps $W^{1,p}(M,N) \cap C^0(M,N)$. 
For such a map $u$, define
\[
 J_p(u)=||du||^p_{sv^p}=\int_M |du|_{sv^p}^p*1.
\]
In order to determine the Euler-Lagrange equations of $J_p$  let
\[
Q(du)^2 :=dudu^T:  T_{u(x)} N \xrightarrow{du^T} T_x M \xrightarrow{du} T_{u(x)} N.
\]
Here $ Q(du)$ is  a section of the bundle 
$End(u^{-1}(TN))$ and 
$|du|_{sv^p}^p=TrQ(du)^p.$
It follows from the multiplication theorems of Sobolev spaces that, in the continuous range $p>n$, $J_p$ is a functional of class $C^{[p]}$, where $[p]$ denotes the largest integer no greater than $p$. The proof of the next proposition is elementary.

%We now express $ Q= Q(du)$ in local coordinates.
%We choose  a local orthonormal frame  $\{e^\alpha \}$ on the tangent bundle of  $U \subset M$ and denote by $\{e_\alpha \}$ the dual frame on the cotangent bundle of $M$. We also fix a local orthonormal  frame $\{F^j \}$  on  the tangent bundle of  $V \subset N$,  $\{F_j \}$  the dual frame on  the cotangent bundle. Assume further that  $u(U) \subset V.$
%Then
%\begin{equation*}
%du=d_\alpha u^i e_\alpha F^i=d u^i F^i; \ dudu^T=d_\alpha u^id_\alpha u^j F^i F_j
%=(d u^i; d u^j)  F^i F_j.
%\end{equation*}
%Recall that $(.;.)$ denotes the inner product on $M$.
% In other words,
%\[
%Q(du)^2_{ij}=(d u^i; d u^j)= Q^2_{ij}, \ \ Q(du)^p=Q^p_{ij}F^i F_j.
%\]
%Equivalently, by viewing $du=d u^i F^i \in \Omega^1(u^{-1}(TN))$, define the dual with respect to metric on $N$
%\begin{equation}\label{Qeqnstarhtro}
%du^\sharp=d u^j F_j \in \Omega^1(u^{-1}(TN^*)).
%\end{equation}
%Then, 
%\begin{equation}\label{Qeqnstarhtrox}
% Q(du)^2=(d u; d u^\sharp)=(d u^i; d u^j)F^iF_j\in \Omega^0(End(u^{-1}(TN))).
%\end{equation}
%We write
%\begin{equation*}
%Q(du)^{p-2}du=Q^{p-2}_{ij}d u^jF^i \in  \Omega^1(u^{-1}(TN))
%\end{equation*}
%\begin{equation*}\label{Qloccord2x}
%*Q(du)^{p-2}du=Q^{p-2}_{ij} *d u^jF_i \in  \Omega^1({u^{-1}(TN)^*}).
%\end{equation*}
%We can now compute the first  variation of $J_p$:

\begin{proposition}\label{prop:firstvar2}The Euler Lagrange equations of the functional  $J_p$ are
\begin{equation}\label{eqn:firstvargen}
\int_M(Q(du)^{p-2}du; D\phi)^\sharp*1=0 \ \ \ \forall \phi \in \Omega^0(u^{-1}(TN)).
\end{equation}
In particular by taking $\phi$  compactly supported away from $\partial M$, 
\begin{equation}\label{eqn:firstvar3}
D^*Q(du)^{p-2}du=0.
\end{equation}
Here $D=D_u$ is the pullback of the Levi-Civita connection on $u^{-1}(TN)$.
\end{proposition}
%\begin{proof} Consider a variation $u=u_t$  with $\frac{du}{dt }= \phi$ and let $J_p=J_p(t)$. Then,
%\begin{eqnarray*}\label{derJ}
%\frac{dJ_p}{dt} &=&\int_M \frac{d}{dt} Tr(dudu^T)^{p/2} *1 \\
%&=& \frac{p}{2}\int_M  Tr (dudu^T)^{p/2-1} D_\frac{d}{dt} (dudu^T) *1 \\
%&=& \frac{p}{2}\int_M  Tr (du du^T)^{p/2-1} ((D_\frac{d}{dt}du du^T)+(du D_\frac{d}{dt}du^T))*1 \\
%&=& p \int_M  Tr (Q(du)^{p-2} du D_\frac{d}{dt} du^T)*1 \\
%&=& p \int_M  ( Q(du)^{p-2} du; D_\frac{d}{dt} du)^\sharp *1  \\
%&=& p \int_M  (Q(du)^{p-2} du;  D \phi )^\sharp  *1.
%\end{eqnarray*} 
%\end{proof}
A critical point of the functional $J_p=||du||^p_{sv^p}$ is called a {\it{Schatten $p$-harmonic map}} or simply a {\it{$J_p$-harmonic map.}} We are mainly interested in the case when $N$ has non-positive curvature. We will show (cf. Corollary~\ref{cor:convex}) that in this case every $J_p$-harmonic map is a minimizer. Otherwise we restrict to minimizers and we call such a map  a {\it$J_p$-minimizing map} or a {\it$J_p$-minimizer}. $J_p$-harmonic maps should not be confused with the usual $p$-harmonic maps which are  critical points of the {\it{different functional}}
\[
||du||_{L^p}^p=\int_M |du|_2^p*1.
\]
The same can be said for the infinity norms.
%This norm {\it{should not be confused}} with the usual $p$-norm of $A$
%\[
% |A|_p= (Tr(A^TA))^{p/2}= \left(\sum_{i=1}^n \lambda_i^2 \right)^{p/2}
%\] 
%except when $p=2$ or either dimension of the vector spaces is equal to one, in which case they agree. 
%Since, as pointed out above, the norms $||.||_p$ are in general unrelated to the $p$-norms $|.|_p$, Schatten $p$-harmonic maps bare no relation  with $p$-harmonic maps. 
The $L^\infty$ norm of $du$ is defined as 
$||du||_{L^\infty}=\it{essup} |du|_2$. {\it In general this is different from the Lipschitz constant which is related to the  the operator norm
$||.||_{sv^\infty}$} (unless one of the dimensions is one).

\subsection{The second variation}
Let $(M, g)$ and  $(N, h)$ as before and $p>n$. We also assume that $t   \mapsto u_t \in W^{1,p}(M, N)$ is a  $C^2$ geodesic homotopy. 
% Thus,  $t  \mapsto J_p(t)$ is twice differentiable and
%\begin{equation}\label{geohom}
%\frac{D}{\partial t}\frac{\partial u}{\partial t} \equiv 0.
%\end{equation} 
We  continue to denote $D=D_u$ the pullback connection on $u^{-1}(TN).$
\begin{lemma} \label{convineq} The following holds:
\[
 (\frac{D}{\partial t}du \ du^T, \frac{D}{\partial t}Q(du)^{p-2} )^\sharp  \geq 0.
\]
\end{lemma}
\begin{proof} The estimate is pointwise. We may choose  normal coordinates near a point so that the  Christoffel symbols vanish at the point. Therefore $\frac{D}{\partial t}=\frac{\partial }{\partial t}$.
For simplicity set $du=A$, $Q^2=AA^T=R$, $A'=\partial A/ \partial t$ and $R'=\partial R/ \partial t$. Then,
\begin{eqnarray*}
p( A'A^T, {Q^{p-2}}' )^\sharp
&=&\frac{p}{2}(R', (R^{{p-2}/2})' )^\sharp
=\frac{p}{2} Tr(R'(R^{p/2-1})')^\sharp\\
&=& D^2 F(R',R')^\sharp
\geq 0. 
\end{eqnarray*}
Here we set $F(R):=TrR^{p/2}$ is convex by the convexity of the Schatten norms. Also in the second equality we used  (\ref{eqn:traceinn}).
%\begin{eqnarray*}
%\lefteqn{<A^T A', (Q^{p/2-1})' >}\\
%&=&trace (A^T A' (Q^{p/2-1})')\\
%&=&trace (A^T A' ((A^TA)^{p/2-1})')\\
%&=& \sum_{i+j=p/2-1}trace ((A^T A')(A^TA)^i((A^T)'A+A^TA')(A^TA)^j )\\
%&=& \sum_{i+j=p/2-1}trace (BC^i(B+B^T)C^j )\\
%&=& \sum_{i+j=p/2-1}trace (BC^i BC^j +BC^i B^TC^j) \\
%\end{eqnarray*}
%where $B=A^T A'$ and $C=A^TA=C^T$.
%\begin{eqnarray*}
%\lefteqn{<\nabla_\frac{\partial}{\partial t} du, du \nabla_\frac{\partial }{\partial t}Q^m>}\\
%&=&trace (du \nabla_\frac{\partial }{\partial t}Q^m \nabla_\frac{\partial}{\partial t} du^T)\\
%&=& trace (\nabla_\frac{\partial }{\partial t}Q^m \nabla_\frac{\partial}{\partial t} (du^T)du)\\
%&=& \sum_{i+j=m-1}trace (Q^i(\nabla_\frac{\partial }{\partial t}Q)Q^j \nabla_\frac{\partial}{\partial t} (du^T)du)
%\end{eqnarray*}
\end{proof}

\begin{proposition}\label{prop:secvar} Let $t   \mapsto u_t \in W^{1,p}(M, N)$ be a  $C^2$ path which is also a geodesic homotopy. Then, 
\begin{eqnarray*}
\frac{1}{p}\frac{d^2J_p}{dt^2}&\geq&-\int_M( R^N \left(Q(du)^{p-2/2}du \frac{\partial u}{\partial t} \right)\frac{\partial u}{\partial t}; Q(du)^{p-2/2}du )^\sharp*1 \\
 &+&\int_M \left |   Q(du)^{p-2/2} D\frac{\partial u}{\partial t} \right |^2 *1.
\end{eqnarray*}
In the above, we view $Q(du)^{p-2/2}du$ and $D\frac{\partial u}{\partial t}$ as  sections of the bundle 
$T^*M\otimes u^{-1}(TN)$. 
\end{proposition}
\begin{proof}
\begin{eqnarray*}
\frac{1}{p}\frac{d^2J_p}{dt^2}&=&\int_M ( \frac{D}{\partial t}D\frac{\partial u}{\partial t};   Q(du)^{p-2}du )^\sharp + (D\frac{\partial u}{\partial t};  \frac{D}{\partial t}(Q(du)^{p-2}du) )^\sharp *1 \\
&=&\int_M( D \frac{D}{\partial t}\frac{\partial u}{\partial t};   Q(du)^{p-2}du)^\sharp*1 -\int_M( R^N\left(du, \frac{\partial u}{\partial t} \right) \frac{\partial u}{\partial t};  Q(du)^{p-2}du)^\sharp*1 \\
&+&
\int_M (D\frac{\partial u}{\partial t};  \frac{D}{\partial t}Q(du)^{p-2} du+  Q(du)^{p-2}\frac{D}{\partial t}du)^\sharp*1 \\
&=&  -\int_M( R^N\left(du, \frac{\partial u}{\partial t} \right)\frac{\partial u}{\partial t};  Q(du)^{p-2}du)^\sharp*1\\
&+&\int_M( \frac{D}{\partial t} du;   \frac{D}{\partial t}Q(du)^{p-2}du)^\sharp*1 \\
&+&\int_M ( D\frac{\partial u}{\partial t};  Q(du)^{p-2}\frac{D}{\partial t}du)^\sharp*1\\
&\geq&-\int_M(R^N (Q(du)^{p-2/2}du, \frac{\partial u}{\partial t} )\frac{\partial u}{\partial t};  Q(du)^{p-2/2}du)^\sharp*1 \\
 &+&\int_M \left |Q(du)^{p-2/2}  D\frac{\partial u}{\partial t}  \right |^2 *1. 
\end{eqnarray*}
In the third equality we used  the geodesic equation $\frac{D}{\partial t}\frac{\partial u}{\partial t} \equiv 0$ and in the last inequality  Lemma~\ref{convineq}.
\end{proof}

\begin{corollary}\label{cor:convex} Assume that $N$ has nonpositive curvature. Then the map 
\[
t \mapsto J_p(u_t)
\]
is convex. In particular, any  critical point is a global minimum.
\end{corollary}

\begin{corollary}\label{cor:unique1} Assume that $N$ has nonpositive curvature. Let $t   \mapsto u_t $ be a  $C^2$  geodesic homotopy between two non-constant minimizing maps $u_0$ and $u_1$. Then,  
\[
\left | \frac{\partial u}{\partial t} \right |  \equiv c \ \ \mbox{and} \ \ \left(R^N \left( Q(du)^{p-2/2}du, \frac{\partial u}{\partial t} \right)\frac{\partial u}{\partial t}; Q(du)^{p-2/2}du \right)^\sharp \equiv 0.
\]
If in addition $\partial M \neq \emptyset$,  then there exists a unique $J_p$-minimizer in a fixed homotopy class.
\end{corollary}
\begin{proof}  Proposition~\ref{prop:secvar} implies
\[ 
Q(du)^{p-2/2}D\frac{\partial u}{\partial t}   \equiv 0,\ \ 
( R^N \left(Q(du)^{p-2/2}du,  \frac{\partial u}{\partial t} \right)\frac{\partial u}{\partial t}; Q(du)^{p-2/2}du)^\sharp \equiv 0. 
\]
We next claim that $D\frac{\partial u}{\partial t} =0$ everywhere.
The first equality implies that $\frac{D}{\partial t}du=D\frac{\partial u}{\partial t} =0$ on the set  $\{du \neq 0 \}$.  On the complement, $du \equiv 0$ and hence the claim also holds. Since $\frac{\partial u}{\partial t} $ is covariantly constant, $|\frac{\partial u}{\partial t}| \equiv c $.
\end{proof}

\begin{corollary}\label{cor:unique2} Assume that $N$ has nonpositive curvature. Let  $t   \mapsto u_t $ be a  $C^2$ geodesic homotopy between two non-constant minimizing maps $u_0$ and $u_1$.  Assuming the target has negative curvature,  either $u_0=u_1$ or the rank of each $u_t$ is $\leq 1$.
\end{corollary}
\begin{proof} If $c=0$ in the above corollary, then $u_0 \equiv u_1$. Otherwise $\frac{\partial u}{\partial t}$ is never zero.
By the negativity of the sectional curvature, $ Q(du)^{p-2/2}du$ must have pointwise rank $\leq 1$ everywhere, hence also $du$.
\end{proof}

\subsection{Existence of $J_p$-minimizers}\label{sect:thpdirichlet}
%Let $(M, g)$  be a compact Riemannian manifold with  boundary $ \partial M$ (possibly empty) and $\dim M=n$. Let $(N, h)$  be  closed and $p>n$. 
The purpose of this section is to show existence of a minimizer $u$ of the functional $J_p$ subject to either  Neumann or Dirichlet boundary conditions. For the Dirichlet problem we fix a continuous map $f:M \rightarrow N$ and seek a minimizer in the homotopy class of $f$ relative to the boundary values of $f$. For the Neumann problem we only fix a homotopy class and no boundary condition at all. The main result of the section is the following:

\begin{theorem}\label{thm:pdirichlet}
Let $(M, g)$  be a compact Riemannian manifold with  boundary $ \partial M$ (possibly empty) and $\dim M=n$. Let $(N, h)$  be closed.   Then, for  $p>n$ there exists a minimizer $u \in W^{1,p}(M, N)$ of the functional $J_p$  with either the Neumann  or the Dirichlet boundary conditions  in a homotopy (resp. relative homotopy) class. 
\end{theorem}
\begin{proof} The proof of existence is fairly standard so  we will only give a sketch. Imbed $N$ in $R^K$ isometrically via the Nash embedding theorem.  Choose a minimizing sequence $u_i$ in the fixed homotopy class of maps to $N$.  This has a weakly convergent subsequence in $W^{1,p}(M,R^K)$ which converges in $C^0$ to a limit $u$ in $W^{1,p}(M, R^K)$. This weak limit $u$ has its image in $N$, is in the same homotopy class as $u_i$, and $J_p(u) \leq \liminf J_p(u_i)$ by lower semi-continuity of $J_p$ on $W^{1,p}(M,R^K)$.  Hence $J_p$ takes on its minimum at $u.$  
\end{proof}

By the Sobolev embedding theorem we also obtain 
\begin{corollary} \label{hoelderpk}The minimizer $u$ of Theorem~\ref{thm:pdirichlet} is in $C^\alpha$ for $\alpha=1-\frac{n}{p}$.
\end{corollary}

\subsection{Variational construction of  the infinity-harmonic map}\label{sect:Varblm}
\begin{definition} Let $(M, g)$  be a compact Riemannian manifold with boundary $ \partial M$ (possibly empty) and $\dim M=n$ and Let $(N, h)$  be a closed Riemannian manifold.
Let $S  \subset M$  and $ f: S \rightarrow  (N, h) $ a  map. We call $f$ $L$-Lipschitz in $S$, if there exists $L>0$  such that for all $x, y \in S$
\[
d_h(f(x), f(y))\leq L d_g(x,y).
\]
The smallest possible $L$, 
\[
L_f(S)= \inf \{L \in \R: d_h(f(x), f(y))\leq Ld_g(x,y) \ \forall x,y \in S \}
\]
is called {\it{the Lipschitz constant}} of $f$ on $S$ and
\[
L_f(x)= \lim_{r \rightarrow 0}L_f(B_r(x)).
\]
the local Lipschitz constant of $f$ at $x $.
\end{definition}

\begin{definition}\label{def:lipdirichlet}
 A map $u \in Lip(M,N)$ is called a {\it{best Lipschitz map}} if $L_u \leq L_f$ for any map  $f \in Lip(M,N)$ homotopic to $u$ (either in the absolute sense or relative to the boundary depending on the context).  
Here $L_u=L_u(M)$ denotes the global Lipschitz constant of $u$  (and similarly for $f$).
%   A best Lipschitz map is called {\it{ an absolute minimizing best Lipschitz map}} if it is also a best Lipschitz map on any Lipschitz domain $\Omega \subset M$ relative to its boundary values.
\end{definition}

The next Lemma follows as in \cite[Lemma 5.3]{daskal-uhlen1} by the upper semicontinuity of the local Lipschitz constant:

\begin{lemma} \label{realized}  For a Lipschitz map $ f: M \rightarrow N$  with global Lipschitz constant $L=L_f(M)$ and a number $0< \mu \leq L$, the set
$ \{ x \in M: L_f(x) \geq \mu \}$ is non-empty and closed. In particular the set of maximum stretch
\[
 \lambda_f=\{ x \in M: L_f(x) =L \} 
\]
is non-empty and closed.
\end{lemma}
%\begin{proof}  First, by the upper semicontinuity of the Lipschitz constant (cf. Proposition~\ref{crandal0}), the set
%\[
% \{ x \in M: L_f(x) \geq \mu \} 
%\]
%is closed. To finish the proof we need to show that $ \lambda_f$ is non-empty.
%Take a sequence $x_i$ such that  $L_f(x_i) \nearrow L$. By compactness, we may assume $ x_i \rightarrow x$ and by upper semicontinuity $L_f(x) \geq \lim_i L_f(x_i) = L$. Thus, $x \in  \lambda_f$ and hence $ \lambda_f \neq \emptyset$. By  upper semicontinuity 
%\[
% \lambda_f=\{ x \in M: L_f(x)= L \}=\{ x \in M: L_f(x) \geq L \} 
%\]
%is closed.
%\end{proof}

%\subsection{The limit of  $J_p$-minimizers as $p \rightarrow \infty$.}
%Let $(M, g)$, $(N, h)$  be as in the previous section. We further assume that $(M, g)$ is convex.  
Fix a homotopy class of maps from $M$ to $N$ (either absolute or relative to the boundary) and choose a Lipschitz map $f: M \rightarrow N$ in that homotopy class. We are going to construct a best Lipschitz map $u:  M \rightarrow N$
homotopic to $f$ as a limit as $p \rightarrow \infty$ of minimizers of the functional
\[
J_p(u)=\int_M Tr Q(du)^p*1 \  \ \mbox{where} \ \ u \in W^{1,p}=W^{1,p}(M, N).
\]

 For $N=\R$ (or maybe better $N=S^1$, since we are assuming the target is compact), this  construction is due to  Aronsson (cf. \cite{aron1},\cite{aron2} and \cite{lind}). See also \cite{daskal-uhlen1}.
  \begin{theorem}\label{thm:existence} 
 Given  a sequence  $p \rightarrow \infty$, there exists a subsequence  (denoted also by $p$) and a sequence of $J_p$-minimizers 
$u_p: M \rightarrow N$   homotopic to $f$ (either absolute or relative to the boundary)
such that   
\[
u=\lim_{p \rightarrow \infty} u_p \ \ \ \mbox{weakly in} \ \ W^{1,s} \ \forall s.
\]
Furthermore, $u_p \rightarrow u$ in $C^0$ and $u$ is a best Lipschitz map in the homotopy class. 
\end{theorem}
 
The proof is similar to the one given in the references above by replacing the $L^p$-norm with Schatten norms. We skip the details.

\begin{definition}\label{def:infhharm} Let $(M, g)$, $(N, h)$  be as before and $u: M \rightarrow N$ a best Lipschitz map in its homotopy class. The map  $u$ is called {\it $\infty$-harmonic} if there exists  a sequence   of $J_p$-minimizers 
$u_p: M \rightarrow N$   homotopic to $u$, {\it called the $p$-approximations of $u$},
such that  
$u=\lim_{p \rightarrow \infty} u_p$
weakly in $W^{1,s}$ for all $s$ (and thus also in $C^0$).
\end{definition}

 \begin{lemma}\label{pintconvto} If $u$ has Lipschitz constant $L$ and $u_p$ is a $p$-approximation,
\[
\lim_{p \rightarrow \infty} J_p^{1/p}(u_p) =  L.
\]
\end{lemma}
\begin{proof}The fact that $u_p$ is a $J_p$-minimizer and Lemma~\ref{lemma:normlim}$(i)$  imply
\[
 J_p^{1/p}(u_p) \leq J_p^{1/p}(u) \rightarrow L.
\]
Hence the $\limsup$ is less than equal to $L$. On the other hand, if $\liminf=a<L$, then proceeding as it the proof of Theorem~\ref{thm:existence}, there exists a Lipschitz map $w$ such that
$L_w\leq a<L$
which contradicts the best Lipschitz constant.
\end{proof}

The following Lipschitz approximation theorem will be needed in Section~\ref{lipqgoesto1} (cf. \cite[Theorem 4.4 and 4.6]{karcher}.)

\begin{theorem}\label{thmgreenewu} If  $f: M \rightarrow N$ is a Lipschitz map, then there exists a sequence
of smooth maps $f_k: M \rightarrow N$ such that $f_k \rightarrow  f$ in $C^0$  and the Lipschitz constants converge,
$L_{f_k} \rightarrow L_f.$
\end{theorem} 

We end the section with the definition of the canonical lamination. 
\begin{definition}\label{canlamin}Let $(M,g)$ and $(N,h)$ be closed hyperbolic surfaces. Given a best Lip map $f: M \rightarrow N$ of $L_f=L>1$, let $\lambda_f$ be its maximum stretch locus (cf. Lemma~\ref{realized}). Let $\mathcal F$ denote the collection of best Lipschitz maps in a homotopy class and set
\[
\lambda=\cap_{f \in \mathcal F} \lambda_f.
\]
We call $\lambda$ {\it{the canonical lamination}} associated to the hyperbolic metics $g$, $h$ and the fixed homotopy class.   
\end{definition}
From \cite{kassel} we know: 

\begin{theorem} 
\begin{itemize}
\item $(i)$  The closed set $\lambda$ is a geodesic lamination on $(M, g)$ and, for any $f \in \mathcal F$,  $f(\lambda)$ is a geodesic lamination on $(N, h)$. Furthermore, for any leaf of $\lambda$, $df$ multiplies arc length by the best Lipschitz constant $L$ (cf. \cite[Lemma 5.2]{kassel}).
\item $(ii)$ There exists $\hat f \in \mathcal F$ such that $\lambda_{\hat f}=\lambda$. We call such a map {\it optimal } (cf. \cite[Lemma 4.13]{kassel}).
\end{itemize}
\end{theorem}

\begin{remark} 
If the homotopy class is that of the identity an easy application of Gauss-Bonet implies that for any best Lipschitz map homotopic to the identity $L_f\geq 1$ and $L_f=1$ iff $f$ is an isometry. Moreover, the canonical lamination is equal to Thurston's chain recurrent lamination $\mu(g,h)$ associated to the hyperbolic structures $g,h$. (cf. \cite[Theorem 8.2]{thurston} and \cite[Lemma 9.3]{kassel}).
\end{remark}

 \section{Conservation laws from the symmetries of the target}\label{sec3}
In the next two sections, we derive formulas for conservation laws given any solution $u_p$ of the Euler Lagrange equations for $J_p.$  Noether's theorem states that for every symmetry of the integrand in a calculus of variations problem (in $\R^n$) and every solution of the Euler Lagrange equations, there exists a divergence free vector field.  In the case of maps between hyperbolic surfaces, the local $so(2,1)$ symmetries of the metric in $N$ and $M$ yield symmetries of the integrand $Tr Q(df)^p *1.$  If we equate vector fields on $M$ with one forms via an area two form, in each case we get a closed $so(2,1)$ valued one form as a conservation law.  Because the symmetries are local, these are sections of flat $so(2,1)$ bundles over $M.$
In this section we derive the conservation laws arising from the symmetry of $N$, and in the next section we derive those arising from the symmetry of $M.$
 We develop the geometry to describe these conservation laws, and we derive the formulas from the geometry, not directly from Noether's theorem.  Admittedly  we would not have found the formulas without being aware of Noether's theorem.

\subsection{The geometry of the hyperboloid}\label{geomhyp}
We review some basic geometry.   Let $e^\sharp  = diag(1,...,1,-1)$.  For $X \in \R^{n,1}$ let    $X^\sharp  = (e^\sharp X)^T$. The inner product in $\R^{n,1}$ is
\[  
( X,Y)^\sharp  =  X^\sharp Y =Y^\sharp X.
\] 
The associated transpose on linear maps $B$  is $B^\sharp  = e^\sharp B^Te^\sharp$.  The group $SO (n,1)$ consists of $(n+1) \times (n+1)$ matrices with determinant one that preserve the inner product. Equivalently, $g \in G$ if $det g=1$ and  $g^{-1} = g^\sharp$.  $B$ is in the Lie algebra $\mathfrak g=so(n,1)$  iff $Tr B=0$ and  $B^\sharp  = -B$.   For $X,Y \in \R^{n,1}$,   $YX^\sharp  - X^\sharp Y$ is a skew symmetric matrix (with respect to $^\sharp$). Define
\begin{equation}\label{wedgedef}
X \times Y=YX^\sharp  - X Y^\sharp \in  \mathfrak g.
\end{equation}
%Note that the above map intertwines left multiplication of $G$ on each copy of $\R^{n,1}$ with the adjoint representation on $Lie(G)$
%so we identify
%\[
%ad(G) \simeq \R^{n,1} \wedge  (\R^{n,1})^\sharp.
%\]
  Let  $(A,B)^\sharp= Tr A B $ denote the Killing form on $\mathfrak g$ 
% With respect to a Cartan decomposition 
% $ \mathfrak g=so(n,1)$,
%  write $A \in \mathfrak g$ as 
% \begin{equation*}
%A = 
%\begin{pmatrix}
%W & X  \\
%X^T & 0
%\end{pmatrix}
%\end{equation*} 
% for $X$ arbitrary vector and $W=-W^T$.
%  It follows
% $(A,A)^\sharp= Tr WW+2|X|^2 $
%  and the metric has 
 which has signature $(n, n(n-1)/2)$. 
 In particular for $n=2$ the Killing form is (up to a constant) a flat Lorenzian metric on $\mathfrak g$ of signature $(2,1)$  isometric to $\R^{2,1}$.
 Let
\[
\HH^n = \HH = \{X \in \R^{n,1}: ( X,X )^\sharp  = -1 \ and \ X_{n+1} \geq 1 \} 
\] 
and $G=SO^+(n,1)$ the index two subgroup of $SO(n,1)$ preserving $\HH$. 
%Let $X_0=(0,..,0,1) \in \HH $ be a conveniently chosen base point. Let $\Pi_0$ be the orthogonal in $(, )^\sharp$  projection onto ${X_0}^\perp$. Let $\Pi_0^\perp$ be the orthogonal projection onto $X_0$.
%Note that
%\[
%  I = \Pi_0 +\Pi_0^\perp= \Pi_0 - (X_0, \cdot )^\sharp X_0 = \Pi_0 -X_0X_0^\sharp,
%  \] 
%If $X \in \HH $ an arbitrary point, then since the action of $G$ on $\HH$ is transitive there exists $g(X) \in  G$ such that $X=g(X)X_0$. Since $G$ acts by isometries it will preserve orthogonal projections on the span of $X$ and the subspace $X^\perp$. In other words
%\begin{lemma}  $g(X) {X_0}^\perp$ is the tangent space $T_X\HH$.
%\end{lemma}

%In what follows, $X_0$ is arbitrary, but it is best to keep one point in mind  $(0,..,0,1)$ at first.  The deck transformations will move $X_0$ around.  The construction is equivariant.  Let ${X_0}^\perp$ be the $<,>^\sharp$-perpendicular space in $\R^{n,1}$ to $X_0$. Then for  $B \in {X_0}^\perp$,  $(\exp B )X_0$ is in  $\HH$. Moreover, given $X \in \HH$ there exists a unique $B \in G$ such that $(\exp B )X_0=X$. We call $g(X) = \exp B$ the unique element of $G$ such that
%\[
%g(X)X_0=X.
%\]
%
%\begin{lemma}  $g(X) {X_0}^\perp$ is the tangent space $T_X\HH$.
%\end{lemma}
For $X$ a point in $\HH$, let $\Pi(X)$ be the orthogonal in $(, )^\sharp$  projection onto $T_X\HH$ and $\Pi^\perp(X)$ be the orthogonal projection onto $X$.  
%Then,     $ g(X)\Pi_0g(X)^\sharp  = \Pi(X)$ is the orthogonal in $^\sharp$  projection from
%$\R^{n,1}$ onto $T_X\HH$.
%\begin{lemma}[Basic structure lemma]  The identification $X \mapsto g(X) \mapsto \Pi(X) $ makes equivalent $\HH$ with a sub manifold of group elements  and with orthogonal in $^\sharp$   projections onto space-like $n$-planes in $adG$.
%\end{lemma}
%Note that
%\[
%  I = \Pi_0 - < X_0, \cdot >^\sharp X_0 = \Pi_0 -X_0X_0^\sharp,
%  \]   
%  hence i.
Then
\begin{equation*}
 \Pi^\perp(X)=-XX^\sharp, \ \Pi(X)=I+XX^\sharp. 
 \end{equation*}
For $v \in \R^{1,n}$,  denote the projection
\begin{equation} \label{proj}
\Pi(X)v=v_X=v+(v,X)^\sharp X.
 \end{equation}
This is similar to the formulas for $S^n$ in $\R^{n+1}$ with the change in sign in $XX^\sharp$ due to the indefinite metric in $\R^{n,1}$.   
  \begin{lemma}\label{metricHemb}The inner product $(,)^\sharp$ restrictred to the tangent bundle of $\HH$ is a Riemannian metric which agrees with the standard metric of the hyperbolic space.
  \end{lemma}
  
  We now discuss the role of infinitesimal isometries.
An element $w$ of the Lie algebra $\mathfrak g$ of $SO(n,1)$ defines a vector field on $\HH$ by setting   $w(X)=wX \in T_X\HH$. Here $X \in \HH \subset \R^{n,1}$ and  $w$ acts on $X$ as a skew-adjoint endomorphism with respect to $(,)^\sharp$. 
Note that $w X \in  T_X\HH$. Indeed,
\begin{eqnarray}\label{infisoacts}
X^\sharp w X &=& (X, w X)^\sharp
= (w X, X)^\sharp \nonumber\\
 &=& (w X)^\sharp X
 = X^\sharp w^\sharp X
 =-X^\sharp w X.
\end{eqnarray}
Alternatively, 
$w(X)=\frac{d}{dt} \Big|_{t=0}e^{tw}X$. 
Let
\[
\alpha_X: \mathfrak g \rightarrow T_X\HH \subset \R^{n,1}, \ \  w \mapsto wX
\]
and
\[
\beta_X: \R^{n,1} \rightarrow \mathfrak g, \ \  v \mapsto v \times X.
\]
\begin{proposition}\label{helpemb0} 
\begin{itemize}
\item $(i)$ For $v \in \R^{n,1}$, $\alpha_X \circ \beta_X(v)=v+(v,X)^\sharp X$. 
\item $(ii)$ For $v \in \R^{n,1}$,
$Tr(\beta_X(v) \beta_X(v))=2 (\alpha_X \beta_X(v),v)^\sharp.$
Thus, the adjoint  $\beta_X^\sharp=2\alpha_X$. 
\item $(iii)$ The map $1/\sqrt 2 \beta_X$ identifies $T_X\HH$  isometrically with its image $\mathfrak p_X \subset ad(E)$.
\item $(iv)$  $\mathfrak g=\ker(\alpha_X)  \oplus  \mathfrak p_X$ is an orthogonal direct sum with respect to the Killing form. The inner product is positive on $\mathfrak p_X$ and negative on $\ker(\alpha_X)$ and corresponds pointwise to a Cartan decomposition of $\mathfrak g$ as a compact Lie algebra and its complement.
\item $(v)$ Given $\xi \in \mathfrak g$  define $\xi_X$ its orthogonal projection onto $\mathfrak p_X$. Then $\xi_X=\beta_X \circ \alpha_X(\xi).$
\end{itemize}
 \end{proposition}
\begin{proof}For $(i)$,
\[
\alpha_X \circ \beta_X(v)=(v \times X) X=Xv^\sharp X-vX^\sharp X=v+(v,X)^\sharp X.
\]
For $(ii)$, let $v \in \R^{n,1}$. Then,
\begin{eqnarray*}  
Tr(\beta_X(v), \beta_X(v))  &=& Tr(X v^\sharp-vX^\sharp)(X v^\sharp-vX^\sharp) \\
 &=& 2 (\alpha_X \beta_X(v),v)^\sharp+2(X,v)^{\sharp 2}\\
 &=&2 (v+(v,X)^\sharp X, v)^\sharp.
\end{eqnarray*} 
In the above we used  $X^\sharp X=-1$ and the fact that $Trvv^\sharp=(v,v)^\sharp$,  $TrXv^\sharp=TrvX^\sharp=(X,v)^\sharp$.

For $(iii)$ note that if $v \in T_X\HH$, then $\alpha_X \circ \beta_X(v)=v$. It follows that $1/\sqrt 2 \beta_X$ restricted to $T_X\HH$ is injective and  identifies $T_X\HH$  isometrically with its image in $\mathfrak g$. 

$(iv)$ is immediate from the above.
To prove $(v)$, first note that if $\xi \in \ker \alpha_X$ then the statement clearly holds. Now suppose that $\xi$ is perpendicular to the kernel of $\alpha_X$. By $(i)$, $\alpha_X(\beta_X(\alpha_X(\xi)))=\alpha_X(\xi)$, thus $\beta_X(\alpha_X(\xi))-\xi \in \ker \alpha_X$. But since both $\beta_X(\alpha_X(\xi))$ and $\xi$ are perpendicular to the kernel of $\alpha_X$, the result follows.
%For $(iii)$ note that  $\alpha_u \circ \beta_u(\alpha_u(\phi))=\alpha_u(\phi)$, thus 
%$\beta_u(\alpha_u(\phi))=\psi+\phi $ where $\psi \in \ker \alpha_u$. Orthogonality follows from the fact that if $\psi u=\alpha_u(\psi)=0$,
%then
%\begin{eqnarray*}
%Tr(\beta_u(\alpha_u(\phi))\psi)&=&Tr(uu^\sharp \phi^\sharp \psi+\phi u u^\sharp \psi)\\
%&=&Tr(\psi uu^\sharp \phi^\sharp -\phi u (\psi u)^\sharp)\\
%&=&0
%\end{eqnarray*}
%$(iv)$ The map $\alpha_u$ must have rank $n$ and it is thus onto $T_u\HH$. Since its complementary subspace is spanned by the vector $u$, the result follows.
\end{proof}

Let $u:M \rightarrow N$ where $N$ is assumed to be hyperbolic but $M$ can be arbitrary. 
More generally, consider a representation $\rho: \pi_1(M) \rightarrow SO^+(n,1)$ and a $\rho$-equivariant map $\tilde u: \tilde M \rightarrow \HH$. For example, $\tilde u$ can be the lift of  $u$ and  $\rho$  the homomorphism induced by $u$ on the fundamental groups. 
Define  the flat bundles on $M$
\begin{equation}\label{exact45231}
 E=\tilde M \times_\rho R^{n,1}, \ ad(E)=\tilde M \times_{Ad(\rho)} \mathfrak g.
\end{equation}
For $u:M \rightarrow N$,  $u^{-1}(TN)$ is a subbundle of $E$ and $du$  a section of $T^*(M) \otimes E$.
%
%
%Since the construction is equivariant, we can also keep track of the connections. 
% Denote still by $d$ the induced flat connection on $E$ by pullback via $u$ (after passing to the quotient) and let
% \begin{equation}\label{exactse20X}
% D_u=u^{-1}(D^\sharp)=\Pi(u) d \Pi(u)+ \Pi^\perp(u) d \Pi^\perp(u)
% \end{equation}
%\begin{equation}\label{exactse20XTY}
%A_u=u^{-1}  A=d\Pi(u) (\Pi(u)  - \Pi^\perp(u)).
% \end{equation}
%
%\begin{lemma}\label{decomphig} $E$ is a flat bundle with flat connection $d$,  $D_u$ is a connection on $E$ and
%\[
%d=D_u+A_u.
%\]
%Furthermore, the restriction of $D_u$ to $u^{-1}(TN)$ (as a subbundle of  $E$ given in (\ref{exactse20})) is  the pullback of the Levi-Civita connection of $N$. The restriction $A_u\Pi(u)$ of $A_u$ to $u^{-1}(TN)$
%is the second fundamental form.
%\end{lemma}
%
%
%\subsection{The infinitesimal symmetries}\label{sectinfis}
%

%Now apply the above construction to a  map $u$ as before. 
%Let $ad(E) $ be the flat Lie algebra bundle over $M$ associated to $\rho$, in other words $ad(E)=\tilde M \times_{Ad(\rho)} \mathfrak g $.
%Recall also  the pullback bundle $u^{-1}(TN) \subset E$. 
The definition of the maps $\alpha$ and $\beta$ before can be modified to include the map $u$.
Since the construction is local we will identify $N$ with $\HH$ and $u$ with $\tilde u$. We view the flat Lie algebra bundle of skew tensors  $ad(E)$
as the space of local isometries. Given $\phi \in ad(E)$,
let $\xi = \phi u \in E$ (in other words, for each $x \in M$ we let $\xi(x) = \phi(x) u(x) \in \R^{n,1}$ For simplicity we drop the dependence on $x$).  By (\ref{infisoacts}),  $\xi=\phi u$ is tangent to $\HH$ at $u$ so there exists a map
\[
\alpha_u: ad(E) \rightarrow u^{-1}(TN) \subset E, \ \  \phi \mapsto \phi u.
\]
Let
\[
\beta_u: E \rightarrow ad(E), \ \  v \mapsto v \times u.
\]
The maps  $\alpha_u$ and $\beta_u$ will play an important role for the rest of the paper. The following follows immediately from Proposition~\ref{helpemb0}:
\begin{proposition}\label{helpemb} 
\begin{itemize}
\item $(i)$ For $v \in E$, $\alpha_u \circ \beta_u(v)=v+(v,u)^\sharp u$. 
\item $(ii)$ For $v \in E$,
\[
Tr(\beta_u(v) \beta_u(v))=2 (\alpha_u \beta_u(v),v)^\sharp.
\]
Thus, the adjoint  $\beta_u^\sharp=2\alpha_u$. 
\item $(iii)$ The map $1/\sqrt 2 \beta_u$ identifies $u^{-1}(TN)$  isometrically with its image $\mathfrak p_u \subset ad(E)$.
\item $(iv)$  $ad(E)=\ker(\alpha_u)  \oplus  \mathfrak p_u$ is an orthogonal direct sum with respect to the Killing form. The inner product is positive on $\mathfrak p_u$ and negative on $\ker(\alpha_u)$ and corresponds pointwise to a Cartan decomposition of $\mathfrak g$ as a compact Lie algebra and its complement.
\item $(v)$ Given $\xi \in ad(E),$ we have $\xi_u=\beta_u \circ \alpha_u(\xi).$
\end{itemize}
 \end{proposition}

\subsection{The Euler-Lagrange equations revisited} 
In this setting, the Noether currents are easily obtained from the Euler Lagrange equations.  We suppose $u: M \rightarrow N$ is in $W^{1,p}$ and 
$\xi $ is in the flat $R^{n,1}$ bundle $E$. We use the notation introduced in (\ref{proj}) to write
\[
        \Pi(u)\xi = \xi_u = \xi + (\xi,u)^\sharp u.
\]
Then $\xi_u$ is a section of  a $W^{1,p}$ sub-bundle $E_u$ of $E$ which is isomorphic to the pullback tangent bundle $u^{-1}(TN)$. Here
\[
      E_u = \{\xi \in E_x: \xi_u(x) = \xi, \forall x \in M \}.
\]
Similarly, for $\psi \in ad(E)$ we denote $\psi_u$ its projection onto the positive definite part $\mathfrak p_u$ in the Cartan decomposition.

If $E^\sharp$ denotes the dual bundle and $W \in T^*(M) \otimes E$,
then
\begin{equation}\label{newfQ}
              Q(W)^2 = (W ;W^\sharp) \in E \otimes E^\sharp
\end{equation}
where $( ; )$ is the inner product in the cotangent space $T^*M$. The tensor $Q^2$ is symmetric in $E \otimes E^\sharp$, and when $W_u = W$, $Q(W)^2$ is non-negative and has a non-negative square root which  sends $E_u$ to itself.

Assume now that $u = u_p$ satisfies the $J_p$-Euler Lagrange equations. Apply these calculations for $Q$ to translate the Euler Lagrange equations obtained in Proposition~\ref{prop:firstvar2} to the present setting.  The formula
\[
             Q(du)^2 = (du;du^\sharp)
\]
is consistent with both descriptions.

   We will also collect some shortcuts in notation:  
\[
<A,B> = \int_M (A; B)^\sharp *1
\]
\[   
 <Q(W)^p> = <Q(W)^{p-2}W,W> = \int_M TrQ(W)^p *1=||W||^p_{sv^p}.
\]
If $\Omega \subset M$, then $<A,B>_\Omega$ refers to the integral over $\Omega$.
Likewise for $<Q(W)^p>_\Omega.$ More generally we will use the notation 
$<f>_\Omega= \int_\Omega f*1$.

\begin{proposition}\label{piati}  In this notation, the Euler-Lagrange equations for $u=u_p$ can be written
\[
 < Q(du)^{p-2}du,  d\xi  > = 0.
\]
for all  $\xi \in \Omega^0(E)$  such that $\xi_u  = \xi$. 
\end{proposition} 

\begin{proof} The Euler-Lagrange equations are
 \[
< Q(du)^{p-2}du, D_u\xi  > = 0.
\] 
where $\xi$ is tangent i.e $\xi_u  = \xi$. By definition,
$D_u\xi =\Pi(u)d\xi=(d\xi)_u $.  But we can drop the $\Pi$ from the formula since we are evaluating it against a tangent vector, $Q(du)^{p-2}du$. 
\end{proof}

Recall the map $\beta_u$ identifying $T\HH$ with a subbundle of the Lie algebra bundle $ad(E)$ and that 
$Q(du)^{p-2}du$ is tangent. In view of the Lemma above it is natural to identify $Q(du)^{p-2}du \in T^*M \otimes T\HH$ with   $(Q(du)^{p-2}du )\times u \in  T^*M \otimes ad(E)$. For simplicity we denote $S_{p-1}(du)= Q(du)^{p-2}du$. 
With this understood, 

\begin{theorem}\label{Prop:Elag-bund}   Let $u=u_p$ satisfy the $J_p$-Euler-Lagrange equations.  Then for any section $\phi \in \Omega^0(ad(E))$, 
\[
 <S_{p-1}(du) \times u, d \phi>= 0.
 \]
 Equivalently, if $V=*S_{p-1}(du) \times u$, then $dV=0$ in the distribution sense.
 \end{theorem}

\begin{proof} By  Proposition~\ref{piati}, since $\phi u$  is automatically in the tangent space of $u$,  for all $\phi \in \Omega^0(ad(E))$
\[
 \int_M  ( Q(du)^{p-2}du; d(\phi u))^\sharp  *1 = 0.
 \]
We work entirely on the integrand
\[
 ( Q(du)^{p-2}du; d(\phi u))^\sharp  =  (Q(du)^{p-2}du; d\phi u )^\sharp  + (Q(du)^{p-2}du;\phi du)^\sharp.
 \]
Rewrite the first term on the left as
\begin{eqnarray*}
 (Q(du)^{p-2}du; d\phi u)^\sharp  &=& Tr (Q(du)^{p-2}du; (d\phi u)^\sharp )\\
&=&Tr( Q(du)^{p-2}duu^\sharp ; d\phi^\sharp )\\
 &=& 1/2 Tr(S_{p-1}(du) \times u; d\phi).
 \end{eqnarray*}
The  last identity follows from the fact that $d\phi$ is skew so we can replace $Q(du)^{p-2}duu^\sharp$ by its skew adjoint part $S_{p-1}(du) \times u=(Q(du)^{p-2}du) \times u$.
The second term on the right can also be rewritten as
\[
 (Q(du)^{p-2}du; \phi du)^\sharp  = Tr(Q(du)^{p-2}du;du^\sharp  \phi^\sharp) =Tr (Q(du)^p \phi^\sharp) = 0.
\]
Here we use the fact that $Q$ is symmetric and $\phi$ skew-symmetric.
\end{proof}

An infinitesimal isometry $\phi$ of $N$ can be viewed a section of $ad(E)$ (or a section of the pullback bundle on the universal cover $\tilde M=\HH$) that is constant, i.e $d\phi=0$ and  $\xi=1/2 \alpha_u (\phi)$ denotes the corresponding (local) Killing field.

\begin{corollary}\label{Elag-bund2}For each infinitesimal isometry $\phi \in \Omega^0(ad(E))$ the 1-form on $M$ given by $\omega_\phi^M=(*S_{p-1}(du),  \alpha_u(\phi))^\sharp$ is closed.
\end{corollary}
\begin{proof}
\begin{eqnarray*}
d\omega_\phi^M &=&d Tr(*S_{p-1}(du) u^\sharp \phi^\sharp)\\
&=&-1/2 dTr(*S_{p-1}(du)\times u \phi) \\
&=&-1/2 Tr(d*(S_{p-1}(du)\times u) \phi)+1/2 Tr(*S_{p-1}(du)\times u d\phi) \\
&=& 0.
\end{eqnarray*}
In the second equality  above we replaced, as in the proof of Theorem~\ref{Prop:Elag-bund}, $*S_{p-1}(du) u^\sharp$ by its skew-symmetric part. In the last equality we used Theorem~\ref{Prop:Elag-bund} and the fact $d\phi=0$.
\end{proof}

The closed 1-form $\omega_\phi^M$ is {\it the Noether current associated to the infinitesimal symmetry $\phi$.} The closedness of the forms $\omega_\phi^M$ is equivalent to Theorem~\ref{Prop:Elag-bund}. One direction is proved in Corollary~\ref{Elag-bund2}. For the converse, let $\tilde V$ denote the lift of $V$ to the universal cover.
Also, pick an orthonormal  basis $\phi _i$ of constant sections of the pullback of the Lie algebra bundle to the universal cover and write
\begin{equation}\label{formuV}
\tilde V=\sum_i (\tilde V, \phi_i)^\sharp \phi_i=\sum \tilde \omega_{\phi_i}^M \phi_i.
\end{equation}
Since $d \phi_i=0$, it follows that $d\omega_{\phi_i}^M=0 \implies d\tilde V=0$ and hence also $dV=0$.

 \subsection{The dual equations}\label{dual}
 In our previous paper, where $N = S^1$, we associated with a solution of the $J_p$-Euler-Lagrange equations of $u = u_p$ a dual (function) $v = v_q$, with 
 \begin{equation}\label{form:conjugate}
 1/p + 1/q = 1.
 \end{equation} 
 The formula was $V=*S_{p-1}(du) = dv$. The limit of the $v$'s,  as $q \rightarrow1$,  turned out to be a key ingredient in the theory, as it defined a transverse measure associated to the maximum stretch lamination.  We follow  this construction in this paper, although the situation is more complicated. 

Recall from the preceding page that if $u = u_p$ is a solution of the $J_p$-Euler-Lagrange equations, then  $Z = Z_q = *S_{p-1}(du) = *Q(du)^{p-2}du \in \Omega^1(E)$ is a 1-form closed with respect to the covariant derivative $D_u Z = (dZ)_u = 0$.   To get a closed 1-form,  let $V = Z \times u$, which is in $\Omega^1(ad E)$.  Now $dV = 0$ and $V = dv$ locally, and the limits of $V = V_q$ as $q \rightarrow 1$ will be the transverse measure. The discussion of these details is left to the next paper. Technically $Z$ and $V$ are now only in $L^q$, although in Section~\ref{sect:regtheor} we show they are in $L^s$ for all $s$. We state the next theorem for both $Z$ and $V$. Computationally they are equivalent.  The difficulties arise, not from the non-linearity, but from the fact that $E $ and $ad(E)$ have indefinite metrics.

We treat $Z  = *S_{p-1}(du)$ first.
 We would like to minimize the functional $J_q$ within the cohomology class of $Z$. In other words, we would like to minimize the $q$-Schatten norm among variations $Z+ d\xi$ for $\xi \in \Omega^0(E)$.  The question is:  What is the optimal choice for $d\xi$?
 
 In order to make sense of measuring $Z + d\xi$, we need to project into a positive definite subspace, so we will measure this by projecting into the tangent space of $\HH$ at $u$, and measure $(Z + d\xi)_u = Z + (d\xi)_u.$   In order to get  a manageable estimate, we will then also to restrict $\xi = \xi_u.$  Note that then $D_u \xi = (d\xi)_u$ is then the covariant derivative induced in the pull back of the tangent bundle of $\HH.$  As before, 
\[
   Q(Z + (d\xi)_u)^2 = (Z + (d\xi)_u; Z+ (d\xi)_u)^\sharp.
   \]

% 
% First note that  if $\sigma$ is an arbitrary variation of $u_p = u $ with respect to a variation of the representation $\rho$ (equivalently if we vary the flat bundle $E$), we still have at  $u$  Euler-Lagrange equations of the form
%\[
% <S_{p-1}(du), d \sigma>=0
%\]
%for some variation $\sigma$.
%However, we cannot integrate by parts, because there will be no cancellation of boundary terms.  On the other hand, we may choose to vary within the cohomology class of $*S_{p-1}(du).$   We get the same variation from using $S_{p-1}(du) + *d\xi$, where $\xi \in \Omega^0(E)$. We may now integrate the expression by parts
%\[ 
%      <*d\xi, d \sigma> =  -<\xi,d d\sigma> = 0.
% \]
%because $\xi$ does satisfy the boundary conditions, and $dd \sigma = 0. $ 
%
%The question is:  What is the optimal choice for $d\xi$?   In order to make sense of measuring $S_{p-1}(du) + *d\xi$, we need to project into a positive definite subspace, so we will measure this by projecting into the tangent space of $\HH$ at $u$, and measure $(S_{p-1}(du) + *d\xi)_u = S_{p-1}(du) + (*d\xi)_u.$   In order to get  a manageable estimate, we will then also restrict $\xi = \xi_u.$  Note that then $D_u \xi = (d\xi)_u$ is then the covariant derivative induced in the pull back of the tangent bundle of $\HH.$  As before, we define:
%\[
%   Q(S_{p-1}(du) + (*d\xi)_u)^2 = (S_{p-1}(du) + (*d\xi)_u;(S_{p-1}(du) + (*d\xi)_u)^\sharp).
%   \]

\begin{theorem}\label{TTheorem A6}   Let $u=u_p$ satisfy the $J_p$-Euler-Lagrange equations and let $Z=*S_{p-1}(du)$. For $\xi \in \Omega^0(E)$ let 
\[
J_q(\xi) =  \int_M Tr Q(Z + (d\xi)_u)^q *1=||Z+ (d\xi)_u)||^q_{sv^q}.
\]
Then the minimum of $J_q$ over all $\xi = \xi_u$ is taken on at $\xi= 0$. Furthermore, 
\begin{equation}\label{eulags0}
D_u *Q(Z)^{q - 2}Z  = 0.
\end{equation}
\end{theorem}
\begin{proof}  Since the computations are done in the pull-back tangent bundle of $\HH$, the induced norm is positive definite here.  By a straightforward extension of the arguments already used, $J_q $ is convex functional. The covariant derivative  $(d |\xi|;d|\xi|)  \leq  (D_u\xi; D_u\xi)^\sharp)$ and $D_u $ has no kernel.  The computation of the Euler-Lagrange equations is also straightforward.  By (\ref{form:conjugate}), 
\[ 
     Q(Z)^{q-2}Z = Q(du)^{(q-2)(p-1)}Q(du)^{p-2}*du = *du.
 \]
Then, $D_udu = 0 $, proving (\ref{eulags0}).
\end{proof}

There is an analogous formulation of Theorem~\ref{TTheorem A6} in terms of  $V=*S_{p-1}(du) \times u$.

\begin{theorem}\label{TTheorem A66}   Let $u=u_p$ satisfy the $J_p$-Euler-Lagrange equations and let $V=*S_{p-1}(du)\times u$.  For $\psi \in \Omega^0(ad(E))$, let 
\[
J_q(\psi) =  \int_M Tr Q(V + (d\psi)_u)^q *1=||V + (d\psi)_u)||^q_{sv^q}.
\]
Then the minimum of $J_q$ over all $\psi =\psi_u$  is taken on at $\psi= 0$. Furthermore, 
\begin{equation}\label{eulags}
D_u*Q(V)^{q - 2}V  = 0.
\end{equation}
\end{theorem}
\begin{proof}
The proof is the same, except that $*(Q(V)^{q-2}V) = du \times u.$  This is the pull-back to $M$ of the Maurer Cartan form on $N.$ We have
\[
         d( du \times u) = du \times du
\]
which is a two form normal to the surface $\HH $ lifted to $ad(E)$ (i.e in the kernel of $\alpha_u$ (cf. Proposition~\ref{helpemb}$(iv)$)) to the surface $\HH $ lifted to $ad(E)$. Hence, $(du \times du)_u = 0.$
\end{proof}

  \section{Conservation laws from the symmetries of the domain}\label{consnoetdom}
In this Section we describe the conservation laws coming from the infinitesimal symmetries of the domain. 
 We first define the energy-momentum tensor $T=T(du)$ associated with the minimizer $u=u_p$ of the functional $J_p$. The results of this section are proved in greater generality  covering a wider class of variational problems for maps between Riemannian manifolds.  The main theorem  is that the energy-momentum tensor  is divergence free. Subsequently, we specialize to the  case of maps between hyperbolic surfaces and the functional $J_p$ that is of interest in this paper. We push $T$ forward to the Lie algebra bundle to obtain, as predicted by Noether's theorem, a {\it closed} $(n-1)=1$-form $W$ with values in the Lie algebra. 
 For an introduction to Noether's theorems see \cite{uhlenNoether}.

\subsection{The energy-momentum tensor}\label{enmomtens}
 In this section $M$ is a compact  $n$-dimensional Riemannian manifold and 
$N \subset \R^k$ a compact {\it isometrically embedded} submanifold. To be consistent with the notation before, we use $(\cdot , \cdot)^\sharp$ for the target metric and $(\cdot ; \cdot)$ for the domain metric.  For every $p \geq 2$,  we define a class of geometric $p$-Lagrangians
\[ 
 F: \R^k \times (\R^k \otimes \R^k) \rightarrow \R^+.
 \]
These should have the properties:
\begin{itemize}
   \item   $(i)$ $F$ is twice differentiable and convex in $X \in \R^k \otimes \R^k$ 
    \item  $(ii)$ For $C,c >0$, $c|X|^{p/2} \leq F(y,X) \leq C(1 + |X|)^{p/2} $
     \item  $(iii)$  $|G_{(i,j)}(y, X)| := |\frac{\partial F(y,X)}{\partial X^{ij}}| \leq 
    |X|^{p/2-1}$ is positive definite.
\end{itemize}

For $f:M \rightarrow N \subset \R^k$, recall the endomorphism $Q^2(df)$ of $End(f^{-1}(TN)).$  By the embedding of $N$ in $\R^k$, this extends to a symmetric $k \times k$-matrix
\begin{equation}\label{defQ}
Q^2_{ij}(df)  = (df^i ; df^j)=g^{\alpha \beta}(d_\alpha f^i ; d_\beta f^j).
\end{equation}
Here $g$ is the Riemannian metric on $M$ and $f^i, i=1,...k$ are the components of $f$.
In this section $p \geq 2$.
  Note that if  $f \in W^{1,p}(M,\R^k)$, $Q^2(df) \in L^{p/2}(M,\R^k)$.  
% Define
% \[
% W^{1,p}(M,N) = \{ f \in W^{1,p}(M,\R^k): f(x) \in N  \ a.e \}.
%\]
We are interested in the variational problem for the integral
\[
         I_p(f) = \int_M F(f, Q^2(df))*1.
\]
%More specifically, if we take 
%$F(y, X)=Tr X^{p/2}$, then $I_p=J_p$ is the functional considered before.

\begin{theorem}\label{genNTheorem 1}Fix $p \geq 2$ and a continuous map $f_0:M \rightarrow N$. There exists $u \in W^{1,p}(M,N)$ such that $I_p(u)$ minimizes 
$I_p(f)$ over all $f \in W^{1,p}(M,N)$ and $u_*=f_{0*}: \pi_1(M) \rightarrow \pi_1(N)$.
\end{theorem}
\begin{proof}  This follows from weak convergence and lower semicontinuity of $I_p$. For a more general result see \cite{white}.
\end{proof}

We are interested in constructing the  momentum tensor for the minimizer $u$. This is a symmetric 2-tensor $T$ in $T^*M \otimes T^*M$ given by
\[
     T(df) =1/p( 2 \sum_{i,j} G_{(i,j)} df^i \otimes df^j - g F).
\] 
 Note that this is well defined as a symmetric two tensor in $L^1$ if $f \in L^p$.
 \begin{remark} The isometric embedding of $N$ into $\R^k$ is used in the above formula. In fact, for an arbitrary metric $h_{ij}$ on $N$ the definition of $T$ is
  \[
     T(df) =1/p( 2 \sum_{i,j,k} h_{jk}(f)G_{(i,j)} df^i \otimes df^k - g F).
 \]
 \end{remark}
 
 \begin{proposition} \label{rep=s0}Assume that the function $F$ is homogenous of degree $p/2$, i.e $F(y, tX)=t^{p/2} F(y, X)$, $t \in \R$. Then 
 $$
Tr_gT(df) =(p-n)/p F(f, Q^2(df)).
$$
 \end{proposition}
 \begin{proof}
 \begin{eqnarray*}
Tr_gT(df) &=&  1/p g^{\alpha \beta} ( 2\sum_{i.j} G_{(i,j)}(f, Q^2(df))d_\alpha f^i d_\beta f^j   - F(f, Q^2(df)) g_{\alpha \beta})\\
&=&  1/p(2\sum_{i.j}  G_{(i,j)}(f, Q^2(df)) (df^i ; d f^j)    - F(f, Q^2(df)) g^{\alpha \beta}g_{\alpha \beta})\\
&=&  1/p(2\sum_{i.j}  G_{(i,j)}(f, Q^2(df)) Q^2_{ij}(df)   - nF(f, Q^2(df))) \\
&=&  (p-n)/p F(f, Q^2(df)).  
\end{eqnarray*}  
 In the last equality we used the homogeneity property of $F$, i.e 
 \  $p/2 \sum_{i,j}  G_{(i,j)}(y, X) X^{ij} \\=F(y, X).$
 \end{proof}
 
 We are primarily interested in the case $F(y, X)=Tr X^{p/2}$, $I_p=J_p$ and the minimum $u$ is the $p$-Schatten harmonic map. Then
  \begin{equation}\label{Slocalcoord1}
T(df) =  (Q(df)^{p-2}df\otimes df)^\sharp - 1/pTr Q(df)^p g 
\end{equation} 
and
\begin{equation}\label{rep=s0}
Tr_gT(df)= (p-n)/p \ TrQ(df)^p.
\end{equation}

Another example is   $F(y, X)=(Tr X)^{p/2}$. In this case, 
$I_p(f)=\int_M |df|^p*1$ 
and the minimum $u$ is the more familiar  $p$-harmonic map. Moreover,
 \begin{equation*}
T(du) =  (|du|^{p-2}du\otimes du)^\sharp - 1/p|du|^p g, \ Tr_gT(du)= (p-n)/p |du|^p. 
\end{equation*}

We now prove the main result of the section that the minimizer $u$ of $I_p$ satisfies $D^*(T(u))=0$. This is to be interpreted in a weak sense. According to the hypotheses, $T(du)$ is only in $L^1$. For all smooth 1-forms $\Phi$ on $M$
\[ 
    \int_M (T(du) ;D\Phi)*1 = 0.
    \]
Note that $D: C^\infty(T^*M)  \rightarrow C^\infty(T^*M \otimes T^*M)$ represents the exterior derivative  on $M$ coupled with the Levi-Civita connection on $T^*M$.
\vspace{.3cm}

 \begin{theorem}\label{genNTheorem 2} If $u=u_p$ is a minimum in $W^{1,p}(M,N)$ of $I_p$, then $D^*{T}(du) = 0$, i.e the symmetric (0,2)-tensor $T=T(du)$ is divergence free with respect to the covariant derivative $D$.
\end{theorem}
\begin{proof}
Let $\phi(t, \cdot)=\phi_t:M \rightarrow M$ be a smooth family of diffeomorphisms such that $\phi_0=id$ and $\frac{d\phi}{dt} \big|_{t=0}  = V$, where $V=\Phi^\sharp$ is the vector field dual to the 1-form $\Phi$.   Then, (omitting the $t$ for now)
\[
I_p(\phi^*u) = 
         \int_M F( u(\phi(x)), Q^2(du  |_{\phi(x)})d\phi |_x)*1.
\]
By a change of variable $y = \phi(x)$ we get
\[
I_p(\phi^* u) = 
\int_M F(u(y),Q^2(du |_yd\phi |_{\phi^{-1}(y)})det(d\phi^{-1} |_y))*1.
\]
Note that the above expression is differentiable in $t$.  Since $u =\phi_0^* u$ in a minimum, the derivative at $t = 0$ must be 0.  Using (\ref{defQ}),  we calculate the derivative at $t = 0$ 
\begin{eqnarray*}
0 &=&
2\int_M [ G_{(i,j)}(u(y),Q^2(du |_y)) (du^i |_y ; D_{d/dt} \big |_{t=0} du^j |_yd\phi |_{\phi^{-1}(y)})\\
     &+&  \frac{d}{dt} \Big |_{t=0} det(d\phi^{-1} |_y)F(u(y),Q^2(du |_y))] *1\\
     &=& 2\int_M [ G_{(i,j)}( u(y),Q^2(du |_y)) (du^i |_y ;  du^j |_y D( \frac{\partial}{\partial t} \Big |_{t=0} \phi(\phi^{-1}(y))))\\
     &+&  \frac{d}{dt} \Big |_{t=0} det(d\phi^{-1} |_y)F(u(y),Q^2(du |_y))] *1\\
 &=& \int_M [ G_{(i,j)}( u(y),Q^2(du |_y))(du^i |_y ; du^j |_y DV |_y)
     - Tr(DV |_y)F(u(y),Q^2(du |_y)] *1\\
     &=& \int_M [ G_{(i,j)}( u(y),Q^2(du |_y))(du^i |_y\otimes du^j |_y ; D\Phi |_y)
     - Tr_g(D\Phi |_y)F(u(y),Q^2(du |_y)] *1.
\end{eqnarray*}
In the second equality we used the fact that $u$ does not depend on $t$ and the identity 
$\frac{D}{dt}d\phi=D\frac{\partial}{\partial t}$.
\end{proof}

%Noether's theorem states that the critical point of an integral which is invariant locally under a group action has a divergence free vector field corresponding to each element of the Lie algebra. In a modern formulation we would say that, if $\xi$ denotes a Killing field, then $*T(\xi)$ is a closed $(n-1)$-form.  
%%\begin{lemma}\label{lemmadivr} For any symmetric (0,2)-tensor $T$ and any smooth vector field $\xi$
%\[
%*D *T(\xi )=*d*T(\xi )-1/2(T; L_\xi g).
%\]
%\end{lemma}
%\begin{proof}
%\end{proof}
%
%The following Corollary is an immediate consequence of  Theorem~\ref{genNTheorem 2} and Lemma~\ref{lemmadivr}: 
%
\begin{corollary}\label{conservformcl} If $M$ has a local symmetry given by the Killing field $\xi: M \rightarrow TM$, and $u$ is a minimum of $I_p$, then the $(n-1)$-form $*T(du)(\xi)$ is closed, i.e
\[
 d*T(du)(\xi) = 0.
\]
In particular, there exists locally a $(n-2)$-form $\theta_\xi$ such that  $*T(du)(\xi)=d\theta_\xi $.
\end{corollary}

\begin{corollary}\label{genNCorollary 4}  If $n = p$, $F$ homogeneous of degree $p/2$ and $u$ is a minimizer for $I_p$, then $T(du)$ is traceless and divergence free. 
%Furthermore,  $I_p$ is conformally invariant, and there is a conservation law corresponding to the conformal symmetry.
\end{corollary}
\begin{proof}  Follows from Theorem~\ref{genNTheorem 2} and (\ref{rep=s0}). 
\end{proof}

\begin{remark} Since a traceless and divergence free symmetric (0,2)-tensor on a Riemann surface is a holomorphic quadratic differential,  Corollary~\ref{genNCorollary 4} should be seen as the higher dimensional generalization of the Hopf differential. 
\end{remark}

\subsection{The case of hyperbolic surfaces}\label{Noethypsurf}
We now specialize to the case where $(M,g)$, $(N,h)$ are hyperbolic surfaces and $u=u_p$ is the minimum of $J_p$ in a given homotopy class. Let $\sigma$ denote the $SO(2,1)$-representation corresponding to the hyperbolic structure on $M$. In analogy with $E$, $ad(E)$   let $F=\tilde M \times_\sigma \R^{2,1}$ and $ad(F)=\tilde M \times_\sigma \mathfrak g$ for the $\R^{2,1}$-bundle and the Lie algebra bundle respectively. 
The maps $\alpha$ and $\beta$ of Proposition~\ref{helpemb0} globalize
\[
\alpha: ad(F) \rightarrow TM \subset F, \ \  \phi \mapsto \phi u
\]
and 
\[
\beta: F \rightarrow ad(F), \ \  v \mapsto v \times id.
\]
Here $id: M \rightarrow M$ is the identity map.

As in Section~\ref{sec3},  $u = u_p$ is a solution of the $J_p$-Euler-Lagrange equations, $S_{p-1}(du) = Q(du)^{p-2}du$ and $V=*\beta_u(S_{p-1}(du)) =(*S_{p-1}(du) \times u$. Set
\[
W=*\beta(T)=*T \times id
\]
where $T=T(du) =  (Q(du)^{p-2}du \otimes du)^\sharp - 1/pTr Q(du)^p g$ is the energy-momentum tensor.  Then $W \in \Omega^1(ad(F))$ of the appropriate Sobolev class. By identifying $T M$ with $T^*M$ via the metric (musical isomorphisms) and $ad(F)$ with its dual, the star operator commutes with $\beta$. Therefore, the order which we take in the definition of $W$ is not important.

Let  $\phi$ be an infinitesimal isometry  of $M$ considered as  a section of $ad(F)$ such that $d\phi=0$.  If  $\xi =1/2 \alpha (\phi)$ denotes the corresponding (local) Killing field,   
  \begin{eqnarray*}
 (W, \phi)^\sharp=(\beta(*T), \phi)^\sharp=2(*T, \alpha(\phi))^\sharp=*T(\xi).
  \end{eqnarray*}
  The closed form $ \omega_\phi^N=(W, \phi)^\sharp=*T(\xi)$ is { \it the Noether current associated to the infinitesimal symmetry $\phi$}.
  
  Lifting $W$ to the universal cover  write, as in (\ref{formuV}), 
  \[
\tilde W=\sum_{i=1}^3*\omega_{\phi_i}^N\phi_i.
  \]
 Thus, the closedness of $W$ is equivalent to the closedness of the Noether currents $\omega_{\phi_i}^N$. In particular, Corollary~\ref{conservformcl} implies
 the following analogue of Theorem~\ref{Prop:Elag-bund} for the symmetries of the domain:
 \begin{proposition} With respect to the flat connection $d$ on $ad(F)$, $W$ is closed in the distribution sense, i.e $dW=0$.
 \end {proposition}
 
 We now record two propositions that will be needed in the future. The first is an expression of the trace of $T$ in terms of $W$. The second gives the relation between $V$ and $W$.
  
 The derivative $d(id)$ of the identity map has values in  $T^*M \otimes TM$ and thus $d(id)\times (id)$ is in $\Omega^1(M, Ad(F))$. We write $d(id)\times (id)(x)=dx \times x=\omega_{mc}(x)$ and note that it is equal to the Mauer-Cartan form. With this notation:
 
 \begin{proposition}\label{positiv1} 
 \[
 *(\omega_{mc}\wedge W)^\sharp=2Tr_g T=\frac{2(p-2)}{p}TrQ(du)^p. 
 \]
 \end{proposition}
 \begin{proof}
% The first equality follows from \cite[Proposition 3.5(iii)]{daskal-uhlen2}. To see the second,
Choose local orthonormal coordinates and write $T=T_{\alpha \beta}dx^\alpha \otimes dx^\beta$. Then, 
 \begin{eqnarray*}
 \lefteqn{(\omega_{mc}\wedge W)^\sharp= 2(d(id)\wedge *T)^\sharp}\\
 &=& 2(dx^\gamma \otimes \frac{\partial}{\partial x^\gamma}\wedge   T_{\alpha \beta} *dx^\alpha \otimes \frac{\partial}{\partial x^\beta})^\sharp\\
 &=&2T_{\alpha \beta} dx^\beta \wedge *dx^\alpha\\
 &=&2Tr_g T *1=\frac{2(p-2)}{p}TrQ(df)^p *1.
\end{eqnarray*}
 The last equality follows from (\ref{rep=s0}).
 \end{proof}

\begin{proposition}\label{prop:relvw} The tensors $W$ and $V$ are related by
\[
-2T=(*V \otimes du \times u)^\sharp; \ \ *W=-2T \times id.
\]
\end{proposition}

 \section{Regularity theory}\label{sect:regtheor}
The study of  $J_p$-harmonic maps is as natural as studying $p$-harmonic maps.   However, we found no references in the literature about Schatten integrals such as $J_p$.  One of the problems is that the usual methods of proving regularity, even the simplest theorem of showing $du_p$ is bounded, do not seem to be applicable.  In Theorem~\ref{mainregtH} we prove the apriori estimates for showing that  for $u_p$ satisfying the  $J_p$-Euler-Lagrange equations,   $D(Q(du_p)^{p/2-1}du_p) \in L^2$. {\it Throughout this section we make the assumption that $M$ and $N$ are hyperbolic surfaces.}

\subsection{Estimates on the indefinite metric}We first derive some basic estimates by considering maps $f: M \rightarrow N$ as sections of the $\HH$-bundle $H$  embedded in the flat $R^{2,1}$ bundle $E$ determined by the homotopy class of the map.  Because the metric on $R^{2,1}$ is indefinite, the geometry is slightly different than what we are accustomed to. These simple results will be used  throughout the subsequent sections.
 
\begin{lemma}\label{distpos} Let $X,Y \in \HH$ and $t = dist (X,Y).$  Then 
\[
t^2 \leq (X - Y,X - Y)^\sharp = 2(\cosh t -1) \leq t^2 \cosh t .
\]
\end{lemma}

\begin{proof}  By equivariance, we may assume $X_i = 0, i = 1,2$,  $X_3 = 1$ and $Y_1 = 0, Y_2 = \sinh t, Y_3 = \cosh t. $ Then distance in $\HH$ from $X$ to $Y$ equals
\[
\int_0^t \left (X'_2(\tau)^2 - X'_3(\tau)^2 \right)^{1/2}d\tau = t. 
\]
%\[
% \int_0^t ((d/d\tau X_3(\tau)^2 - (d/d\tau X_2(\tau))^2)^{1/2}d\tau) = \int_0^t (\cosh^2(\tau) - \sinh^2(\tau))^{1/2} d\tau = t.
%\]
The geodesic is 
\[
     X_1(\tau) = 0, \ X_2(\tau) = \sinh(\tau) \  X_3(\tau) = \cosh(\tau).
\]
Since 
\[
(X-Y,X-Y)^\sharp = (\sinh t)^2 - (\cosh t - 1)^2 = 2(\cosh t - 1)\geq
 t^2,
 \]
 the estimate $2(\cosh t - 1) \leq t^2 \cosh t $ can be obtained via calculus.
\end{proof}

Denote the orthogonal projection of $W \in R^{2,1}$ into the tangent space of $\HH$ at $X$ 
by 
\begin{equation}\label{proj1}
W_X =\Pi(X)W= W + (W,X)^\sharp X.
\end{equation}
\begin{lemma}\label{supptLemma2} Let $X, Y \in \HH$, $W_Y = W$.  Then
\[
  (W,W)^\sharp \leq (W_X,W_X)^\sharp \leq \left (1 + 1/2(X-Y,X-Y)^\sharp \right)^2(W,W)^\sharp.
 \]
\end{lemma}
\begin{proof}
\[
       (W_X,W_X)^\sharp = (W + (W,X)^\sharp X, W+(W,X)^\sharp X)^\sharp =
        (W,W)^\sharp + {(W,X)^\sharp}^2.
\]
This implies the left hand side inequality. Moreover, 
\begin{eqnarray}\label{strineq1}
   {(W,X)^\sharp}^2 &=& {(W, X - Y)^\sharp}^2 
    = {(W,(X -Y)_Y)^\sharp}^2 \nonumber \\
     &=& {(W, X - Y + (X - Y,Y)^\sharp Y)^\sharp}^2 \nonumber\\
 &\leq& (W,W)^\sharp \left((X - Y, X - Y)^\sharp + {(Y-X,Y)^\sharp}^2\right) \\
 &=& (W,W)^\sharp (X - Y, X - Y)^\sharp (1 + 1/4(X - Y, X  - Y)^\sharp).\nonumber
  \end{eqnarray}
The inequality  follows from the fact that both terms are in the tangent space to $Y$ in which the inner product is positive definite, and the last equality from
\begin{eqnarray*} 
    (Y - X,Y)^\sharp &=& 1/2(2(-X,Y)^\sharp -2)
     = 1/2(-2(X,Y)^\sharp +(Y,Y)^\sharp +(X,X)^\sharp) \\
     &=&1/2(X - Y, X - Y)^\sharp.
 \end{eqnarray*}
 The right hand side  inequality follows from this.
%Now we have the bound of 
%\begin{eqnarray*}
%<W_X, W_X>^\sharp  &\leq&  <W,W>^\sharp(1 + (X - Y),(X-Y)>^\sharp \\
%&+ &1/4<(X - Y),(X - Y)>^\sharp
% \end{eqnarray*} 
 \end{proof}
 
We assume $W: T_x \tilde M=T_x \HH \rightarrow T_X \HH \subset R^{2,1} $. Typically $W$ will be the differential of the lift of  a map between two hyperbolic surfaces. Recall from (\ref{newfQ}) that we defined $Q(W)$ by
    $Q(W)^2 = (W ; W^\sharp) $. 
 Let $s(W) = s_1(W)$ be the largest eigenvalue of $Q(W) $ (or singular value of $W$) as in Definition~\ref{schatten}. Then,  
\begin{eqnarray}\label{shatl1}
       s (W)^l \leq Tr Q(W)^l \leq 2s (W)^l
 \end{eqnarray}
which will make it useful in estimates.
At the same time, we find the weight 
\begin{eqnarray}
      \omega(X,Y) = 1 + 1/2(X-Y, X-Y)^\sharp \geq 1 
 \end{eqnarray}\label{deromea}
appears often in our estimates ($ \omega \geq  1$ by  Lemma~\ref{distpos}).

%\begin{lemma} \label{compprojnorm} If $U = U_Y$, then 
%\[
%      s (U)\leq s (U_X)\leq \omega(X,Y) s(U).
%\]
%In fact, this is true for both eigenvalues:
%\[
%      s_j(Q(U)) \leq s_j Q((U_X))  \leq \omega(X,Y)s_j(Q(U)).
%      \]
%\end{lemma}
%\begin{proof}This follows from the estimates in Lemma~\ref{supptLemma2}  and the definition of $\omega(X,Y).$ {\bf{Check}}
%\end{proof}

\begin{corollary} \label{compprojnorm} If $W=W_Y: T_y \HH \rightarrow T_Y \HH \subset R^{2,1} $, then the pointwise Schatten norms satisfy
\[
   |W_X|_{sv^p}  \leq \omega(X,Y)|W|_{sv^p} \ \ 1 \leq p \leq \infty.
      \]
      In particular,
\[
       s(W_X)  \leq \omega(X,Y)s(W).
      \]
\end{corollary}
\begin{proof} Lemma~\ref{supptLemma2} implies that $|W_X| \leq \omega(X,Y)|W|$ pointwise. The rest follows immediately from   Lemma~\ref{ineqsvp}.
\end{proof}
%In the following, $w$ is in the homotopy class of  maps $M \rightarrow N$ we fixed from the beginning, so we can regard it as a section of the fiber bundle $H $ in our fixed flat $R^{2,1} $ bundle $E$ defined in (\ref{exact452}). Here $f:M \rightarrow N$ is a comparison map. 

%Assume that $W_1, W_2: T_xM \rightarrow T_X \HH$.  Examples would be $W_1 =dw_1$ and $W_2 =dw_2$  for two sections of $H$. We write $(W_1;W_2)^\sharp$ for the inner product in $T^*M \otimes R^{2,1}$. When $W_i$ are in $T^*M \otimes w^{-1}T(\HH)$ the inner  product is positive definite. 
%We do not normalize the measures yet.  For the time being $W = dw.$  This is because we will want to normalize differently in different regions of $M.$

%\begin{lemma}  Assume $w, f :M \rightarrow N$, $W =dw$. Let  $(w - f)_w = w - f + R$, where  $R = <w - f,w>^\sharp w. $  Then 
%\[
%         0 \leq <Q(W)^{p-2}W,dR>^\sharp = 1/2<w-f,w-f>^\sharp <Q(W)^{p-2}W;W>^\sharp.
% \]
%\end{lemma}
%\begin{proof}  Note $W= W_w$ and $Q(W)^{p-2}W = (Q(W)^{p-2}W)_w.$ Hence 
%\[
%<Q(W)^{p-2}W,d R>^\sharp = <Q(W)^{p-2}W, (dR)_w>^\sharp.
%\]
% Using the fact that $w$ is orthogonal  to the tangent space and $<w,w>^\sharp = <f,f>^\sharp = -1$,  
% \begin{eqnarray}\label{formdRw}       
%  (d R)_w =  <w-f,w>^\sharp dw = 1/2<w-f,w-f)>^\sharp dw.
%  \end{eqnarray} 
%\end{proof}

\begin{lemma}\label{formdRw}  Assume $w, f :M \rightarrow N$, $W =dw$. Let  $(w - f)_w = w - f + R$, where  $R = (w - f,w)^\sharp w. $  Then 
 \begin{eqnarray*}       
  (dR)_w = 1/2(w-f,w-f)^\sharp dw.
  \end{eqnarray*}
%  and
% \begin{eqnarray}\label{formdRw2}
%         0 \leq (Q(W)^{p-2}W,dR)^\sharp = 1/2 \phi^2(w-f,w-f)^\sharp Tr Q(W)^p.
% \end{eqnarray}
\end{lemma}
\begin{proof}  
Using the fact that $w$ is orthogonal  to the tangent space,
\begin{eqnarray*}
(dR)_w&=&\left (d(w - f,w)^\sharp w+(w - f,w)^\sharp dw \right)_w 
=(w - f,w)^\sharp dw.
\end{eqnarray*}
Since $(w,w)^\sharp = (f,f)^\sharp = -1$,  
 \[       
  (w-f,w)^\sharp = 1/2(w-f,w-f))^\sharp.
\] 
The lemma follows immediately from this. 
%Since
%$W= W_w$ and $Q(W)^{p-2}W = (Q(W)^{p-2}W)_w$, 
%\[
%(Q(W)^{p-2}W,d R)^\sharp = (Q(W)^{p-2}W, (dR)_w)^\sharp.
%\]
%Combined with (\ref{formdRw}),
%the second formula follows. (Note the positivity follows from Lemma~\ref{distpos}).
\end{proof}

%Before we move on to the next estimate, we prove a basic estimate which is based on the convexity of the point wise Shatten norms $||\cdot||_p$.  It can be proved in coordinates, but this is a more elegant approach.  In the second part of the lemma, note that $W_w = W$ and $F_w = F$. So the estimates are made pointwise in the tangent space to $N $ at $w,$ on which the inner product is positive definite. When we proved the convexity of $J_p,$ it was necessary to use the negative curvature of $N.$  Here we regard $W(x)$ as a matrix mapping $T_xM$ to $T_{w(x)}N$ and the curvature of $N$ does not enter.

%\begin{lemma}\label{Propo7.4-1}Assume $w = u_p$ satisfies the $J_p$-Euler-Lagrange equations, $W=du_p$, $f$ is a comparison map and $\phi^2$ is a non-negative cut-off function.  Then
%\[
%<Q(W)^{p-2}W, d(\phi^2(w-f))> = 
%        - 1/2<\phi^2 (w-f,w-f)^\sharp Tr Q(W)^p >  \leq 0. 
% \]
%\end{lemma}
%\begin{proof} 
%\begin{eqnarray*}
%\lefteqn{<Q(W)^{p-2}W, d(\phi^2(w-f))>}\\
%&=&<Q(W)^{p-2}W, d(\phi^2(w-f))_w>-<Q(W)^{p-2}W, d(\phi^2R)>\\
%&=&<Q(W)^{p-2}W, D_{w}(\phi^2(w-f))_{w}>-1/2 <\phi^2(w-f,w-f)^\sharp Tr Q(W)^p>\\
%&=&-1/2 <\phi^2(w-f,w-f)^\sharp Tr Q(W)^p>.
%  \end{eqnarray*}
%Here we used that $w$ satisfies the Euler-Lagrange equations for $ J_p$ and (\ref{formdRw}).
%\end{proof}
%
\begin{proposition}\label{supptprop5}  Assume $w = u_p$ satisfies the $J_p$-Euler-Lagrange equations, $W=du_p$,  and $f :M \rightarrow N$ is a comparison map.  Then
\[
   < \omega(w,f) S_{p-1}(W), W - F> = 0.
\]
Here $S_{p-1}(W)=Q(W)^{p-2}W$, and $F = \omega(w,f)^{-1}(df)_{w} $ satisfies $s(F) \leq s(df).$
\end{proposition}
\begin{proof}  Because $w$ satisfies the Euler-Lagrange equations for $ J_p$, 
\begin{eqnarray*}
    0 &= &<S_{p-1}(W), D_{w}(w-f)_{w}> = <S_{p-1}(W), d(w-f)_{w}> \\
    &= &<S_{p-1}(W), W - df + dR> = <S_{p-1}(W), \left(1 + 1/2 (w-f, w-f)^\sharp \right)W - df> \\
    &= &<\omega(w,f)S_{p-1}(W), W - F>.
     \end{eqnarray*}
  Here we have used that    $S_{p-1}(W)=Q(W)^{p-2}W$ is tangent, and  Lemma~\ref{formdRw}.
%Here  we have moved the scalar quantity around, and defined $F  = (df)_w \times \omega(w,f)$
  The fact that $s(F) \leq s(df)$ follows from  Corollary~\ref{compprojnorm}. \end{proof}

The maps we are dealing with are not necessarily smooth. We can approximate $W^{1,p}$ maps by $C^\infty$ maps, prove the inequalities for them (of course, remembering that the approximations do not satisfy the Euler-Lagrange equations, but do up to a term which goes to zero as the approximation goes to its limit).  This is a standard technique in basic differential topology of maps based on Banach norms, and we do not go into the details.

The next Lemma follows immediately from Lemma~\ref{lemmaconvvnorms}. We simply use the pointwise inequality  on differentiable sections $W$ and $F$ of the tangent bundle with a smooth weight.  The weights simply multiply the inequality point wise. Since the differentiable sections  and smooth weights are dense, it follows for all $W^{1,p}$  sections $W$ and $F$ and bounded measurable non-negative weights $\omega.$

\begin{lemma}\label{sptlemma7}  $w: M \rightarrow N$ a $W^{1,p}$ map, $W=dw$.  If  $0 \leq \omega=\omega(x) \leq k$ is a weight and  $F_w=F$, then
% $W_w = W$, $F_w = F$ then
\[
p<\omega Q(W)^{p-2}W, F - W> \leq <\omega Tr(Q(F)^p-Q(W)^p)>.
\]
\end{lemma}

\subsection{Pointwise inequalities} In this section we will prove some pointwise inequalities that we will need in the sequel. We found no reference for similar inequalities in the literature. Let $V_1$ and $V_2$ be inner product spaces of dimension 2 and $A,B \in Hom(V_1, V_2)$.   In the applications of the first inequalities $A = du$ and $B = (dw)_u$ both map  $V_1=T_xM$  to $V_2=T_{u(x)} N$.  As before  $Q(A)^2 = A  A^T$ is a symmetric map of $V_2$ and $\mathcal Q(A)^2 = A^T  A$  a symmetric map of $V_1$. Let $S_q(A)=Q(A)^{q-1}A$  a map from $V_1$ to $V_2$. We also introduce the notation $\delta(X,Y)=(X-Y, X-Y)^\sharp.$

 We start with the following non-standard inequality:

\begin{lemma}\label{lemmappp}  For $x, y > 0$, $p>2$
\[
       (x^{p/2} \pm y^{p/2})^2 < p(x^{p-1} \pm y^{p-1})(x\pm y). 
       \]
\end{lemma}
\begin{proof}  Assume $x\geq y$ and note the inequality is trivial if $y = 0.$ If $y > 0$, divide the expression by $y^p.$ We see it is sufficient to prove the inequality for $x' = x/y\geq 1$, $y' = 1$. Prove this first for the minus sign, which is less standard. At $x = 1$, both sides and their derivatives are zero. If we compute the second derivatives of each side, we get on the right
\[
     p\left(p(p-1)x^{p-2} - (p-1)(p-2)x^{p-3}\right) = p(p-1)x^{p-3}(px- (p-2)) \geq     p(p-1)x^{p-2}.
     \] 
Note we used $x\geq1$ for this last step.
And on the left
\[
         p(p-1)x^{p-2} - p(p/2 - 1)x^{p/2 - 2} < p(p-1)x^{p-2}.
         \]
So the second derivative of the right-hand expression is less than the second derivative on the left, and we can concluded the inequality.
     For the $+$ sign, we simply write out left and right sides. Assuming $p>2$,
 \begin{eqnarray*}
 (x^{p/2} + 1)^2 = x^p +2 x^{p/2} + 1\leq x^p +2 x^{p-1} + 1\\
 \leq  2(x^{p-1} + 1)(x+1) < p(x^{p-1} + 1)(x+1).
\end{eqnarray*}
\end{proof}

\begin{proposition} \label{propineq1}  
\begin{eqnarray*}
(S_{p/2}(A) - S_{p/2}(B); S_{p/2}(A) - S_{p/2}(B))^\sharp \leq p (S_{p-1}(A) - S_{p-1}(B), A- B)^\sharp.
\end{eqnarray*}
\end{proposition} 

\begin{proof} Since the diagonalizable matrices  are dense, it suffices to prove the inequality for those. We will do this for all pointwise inequalities in this section without explicit mention each time.

 Let $A = \sum \alpha_j a_j \otimes A_j$ and $B = \sum \beta_j b_j \otimes B_j$   where $a_j$ and $b_j$ are orthonormal bases for $V_1$,  $A_j$ and $B_j$ are orthonormal bases for $V_2$ chosen so the real numbers $\alpha_j$ and $\beta_j$ are positive. Then we can write
 \begin{eqnarray*}
       b_1 &=& \cos \theta a_1 + \sin \theta a_2 \ \ \ b_2 = -\sin \theta a_1 + \cos \theta a_2\\
        B_1 &=& \cos \phi A_1 + \sin \phi A_2 \ \ B_2 = -\sin \phi A_1 + \cos \phi A_2.
\end{eqnarray*}
A direct computation gives 
 \begin{eqnarray*}
      \lefteqn{ (S_{p-1}(A) - S_{p-1}(B); A-B)^\sharp} \\
      &=&(Q(A)^{p-2}A - Q(B)^{p-2}B; A-B)^\sharp\\
    &=& (\sum_k (\alpha_k^{p-1}a_k \otimes A_k - \beta_k^{p-1}b_k \otimes B_k);  \sum_j (\alpha_j a_j \otimes A_j - \beta_j b_j \otimes B_j))^\sharp\\
   &=& \sum_j ( \alpha_j^p +\beta_j^p - 
      \cos \theta \cos \phi(\alpha_j^{p-1}\beta_j + \beta_j^{p-1}\alpha_j) -
       \sin \theta \sin \phi(\alpha_j^{p-1}\beta_{j'} + \beta_j^{p-1}\alpha_{j'})
 \end{eqnarray*}
Here $1' = 2$ and $2' = 1$.  Let 
\[ 
E_1 = (-1)^n \cos \theta\cos \phi \geq 0 \ \mbox{and} \  E_2 = (-1)^m \sin \theta\sin \phi \geq 0.
\]  
Note that $E_1 + E_2 \leq 1$.  We rewrite
 \begin{eqnarray*}
     \lefteqn{(S_{p-1}(A) - S_{p-1}B; A-B))^\sharp} \\
     &=& \sum_j  (\alpha_j^p + \beta_j^p) (1 - E_1 - E_2)
    + E_1 (\alpha_j^{p-1} - (-1)^n \beta_j^{p-1})(\alpha_j -(-1)^n\beta_j)\\
&+&E_2 (\alpha_j^{p-1} - (-1)^m \beta_{j'}^{p-1})(\alpha_j -(-1)^m \beta_{j'}).
 \end{eqnarray*}
Using the same rules, we get
 \begin{eqnarray}\label{Esecondterm1}
      \lefteqn{\left |S_{p/2}(A)-S_{p/2}(B) \right|^2}\nonumber \\
      &=&\sum_j ( \alpha_j^p + \beta_j^p )(1 - E_1 - E_2)
    + E_1 (\alpha_j^{p/2} - (-1)^n \beta_j^{p/2})^2  \\
&+&E_2 (\alpha_j^{p/2} - (-1)^m \beta_{j'}^{p/2})^2. \nonumber
 \end{eqnarray}
The first terms are equal in the two expressions, so the insertion of $p$ increases the right hand side.  Term by term, the inequality follows by applying Lemma~\ref{lemmappp}  to the pairwise expressions with the same indices in the sums which multiply the E's.
\end{proof} 

The next inequality is much easier and follows from the same computation. The relevant inequality in one variable is the following simple: 

\begin{lemma}\label{lemmappp234}For $x, y > 0$, $p>2$ and $z \geq \max\{x,y\}$
\[
       (x^{p-1} \pm y^{p-1})^2 < 4z^{p-1}(x^{p/2} \pm y^{p/2})^2. 
       \]
\end{lemma}
\begin{proof}Assume without loss of generality that $x \geq y$. Divide both sides by $y^{p-1}$, so it suffices to prove
\[
       (x^{p-1} \pm 1) \leq 2z^{p/2-1}(x^{p/2} \pm 1) \ \ \mbox{for} \ x \geq 1.
 \]
 The inequality holds for $x = 1.$ Since the derivative of the left hand side is less than the derivative of the right hand side the inequality follows.
\end{proof}

\begin{proposition}\label{propineq2}
 \begin{eqnarray*}
\left |S_{p-1}(A) - S_{p-1}(B) \right|^2 \leq
      4(\max(s(A),s(B))^{p-2}\left | S_{p/2}(A)-S_{p/2}(B) \right|^2.
       \end{eqnarray*}
\end{proposition}
 \begin{proof} The left hand side of this inequality is the same as the one in (\ref{Esecondterm1}) with $p$ replaced by $2p-2$. Thus,
\begin{eqnarray}\label{Esecondterm2}
\lefteqn{\left |S_{p-1}(A) - S_{p-1}(B) \right|^2} \nonumber \\
     &=& \left |Q(A)^{p-2}A - Q(B)^{p-2}B \right|^2 \nonumber\\
      &=&\sum_j ( \alpha_j^{2p-2} + \beta_j^{2p-2} )(1 - E_1 - E_2)
    + E_1 (\alpha_j^{p-1} - (-1)^n \beta_j^{p-1})^2\\
&+&E_2 (\alpha_j^{p-1} - (-1)^m \beta_{j'}^{p-1})^2.\nonumber
 \end{eqnarray}
 If we multiply the terms of four times the expression of $|Q(A)^{p/2-1}A - Q(B)^{p/2-1}B|^2$ given in (\ref{Esecondterm1}) by the largest of  the four terms $\alpha_k^{p-2}$ and $\beta_k^{p-2}$, $k = 1,2$,  this  will dominate the  corresponding terms of $|Q(A)^{p-2}A - Q(B)^{p-2}B)|^2$ given in (\ref{Esecondterm2}). 
\end{proof}

In the next inequality, we look at a slightly more complicated situation. Let
\[
\tilde B: V_1 =T_xM \rightarrow \tilde V_2 = T_YN \subset R^{2,1}
\]
 where $Y = w(x)$ is a different point than $X = u(x).$ Let 
 \[
       B = \tilde B_X = \tilde B + (\tilde B,X)^\sharp X =\tilde B + (\tilde B, X-Y)^\sharp X.
       \]
Note that $(\tilde B, X)^\sharp$ is a cotangent (= tangent) vector on $M$ 
defined by
\[
(\tilde B,X)^\sharp(a)=(\tilde B(a),X)^\sharp \ \mbox{where} \ a \in V_1=T_xM.
\]
Also (after identifying tangent with cotangent vectors)
\begin{eqnarray}\label{mathcalQproj}
   \mathcal Q(B)^2 &= & \mathcal Q(\tilde B)^2 + (\tilde B, X)^\sharp \otimes (\tilde B, X )^\sharp  \\
   &= & \mathcal Q(\tilde B)^2 + (\tilde B, X-Y)^\sharp \otimes (\tilde B, X-Y )^\sharp. \nonumber
 \end{eqnarray}
 Indeed, for $a, c \in V_1=T_xM$,
 \begin{eqnarray*}
 (\mathcal Q(B)^2a;c) &=&(Ba, Bc)^\sharp \\
 &=& (\tilde Ba + (\tilde Ba, X)^\sharp X, \tilde Bc + (\tilde Bc, X)^\sharp X)^\sharp \\
 &=& (\mathcal Q(\tilde B)^2a;c)+(\tilde Ba, X)^\sharp (\tilde Bc, X)^\sharp.
 \end{eqnarray*}

\begin{proposition} \label{morecomp1} If $\tilde B_Y = \tilde B$ (namely, $ \tilde B $ is in the tangent space of $N$ at $Y$), then   
\[
((S_{p-1}(\tilde B),X)^\sharp;  (\tilde B,X)^\sharp) \leq 
             Tr Q(\tilde B)^p \delta(X,Y)(1 + 1/4\delta(X,Y))
            \]
\end{proposition}
\begin{proof}  We can decompose $\tilde B = \sum_j \tilde \beta_j b_j \otimes \tilde B_j$ where $b_j$ and $\tilde B_j$ are orthonormal bases of $V_1=T_xM$ and  $V_2 = T_YN$ respectively. Then
\begin{eqnarray*}
     ((S_{p-1}(\tilde B),X)^\sharp;(\tilde B,X)^\sharp )&=& 
\sum_j \tilde \beta_j^{p-1}b_j(\tilde B_j,X)^\sharp \sum_k \tilde \beta_k b_k (\tilde B_k,X)^\sharp \\
    &=& \sum_j \tilde \beta_j^p {(\tilde B_j,X)^\sharp}^2 \\
    &\leq&\sum_j \tilde \beta_j^p{(\tilde B_j, \tilde B_j)^\sharp}^2\delta(X,Y)(1 + 1/4\delta(X,Y))\\
    &=& \sum_j \tilde \beta_j^p\delta(X,Y)(1 + 1/4\delta(X,Y)).
\end{eqnarray*} 
In the above the  inequality follows from (\ref{strineq1}).
\end{proof}
%\begin{eqnarray*}   
%   (\sup(s(A),s(B))^{p-2}((Q(A)^{p/2-1}A-Q(B)^{p/2-1}B);
%      (Q(A)^{p/2-1}A-Q(B)^{p/2-1}B)).
%      \end{eqnarray*} 
%unction, then
%\begin{eqnarray*}
%<Q(W)^{p-2}W, d(\phi^2(w-f))> = 
%        - 1/2<\phi^2 <w-f,w-f>^\sharp trace Q(W)^p >  
%\end{eqnarray*}

Now we come to the most complicated situation.  The terms we are estimating do not appear in the final computation of the derivative from difference quotients, so they are smaller than our other terms. Because $p$ is large, it is a real nuisance to compute.  We do an easy computation to warm up.  We could get better decay for $p \geq 8,$ but we give the proof which includes $p = 4.$ We assume that $\delta(X,Y) = (X-Y, X-Y )^\sharp < 1/10$, so as not to carry around an extra factor.

\begin{lemma}\label{props5}   Let  $ \mathcal   Q(B)^{2q} - \mathcal Q (\tilde B)^{2q}: V_1 \rightarrow V_1$ be a symmetric map, with $\tilde B_Y = \tilde B$ and $\tilde B_X = B.$  Then the eigenvalues of $\mathcal   Q(B)^{2q} - \mathcal Q (\tilde B)^{2q}$ are non-negative and bounded above by $2q\beta_1^{2q}\delta(X,Y).$
\end{lemma}
\begin{proof} 
Recall from (\ref{mathcalQproj}), $\mathcal Q  (\tilde B)^2 = \mathcal Q(B)^2 - (\tilde B,X)^\sharp \otimes (\tilde B,X)^\sharp.$
 We use (\ref{strineq1}) to estimate
 \begin{equation}\label{normtilBX}
| (\tilde B,X)^\sharp |  
         \leq s(\tilde B)(\delta(X,Y)(1+ 1/4\delta(X,Y)))^{1/2}.
\end{equation}
We rewrite  
\[
    (\tilde B,X)^\sharp \otimes (\tilde B,X)^\sharp = s(B)^2 \delta(X,Y) C
\]
where $C$ is a rank 1 symmetric matrix whose entries are all less than a number slightly larger than 1 (recall $\delta(X,Y)$ is small).

 We want to compute the operator norm of
$\mathcal Q(B)^{2q} - \mathcal   Q(\tilde B)^{2q}.$
As symmetric matrices, 
\begin{eqnarray*}
 0 \leq \mathcal  Q(B)^{2q} -\mathcal   Q(\tilde B)^{2q} &=&
   \mathcal  Q(B)^{2q} - (\mathcal   Q(B)^2 - s(B)^2\delta(X,Y)C)^q \\
   &\leq &
   (\mathcal   Q(B)^{2q} - (\mathcal   Q(B)^2 - 2s(B)^2\delta(X,Y)Id)^q.
\end{eqnarray*}
The eigenvalues of the larger matrix are
\[ 
       \beta_j^{2q} - (\beta_j^2 - 2\beta_1^2\delta(X,Y))^q \leq
       2q \beta_1^{2q} \delta(X,Y).  
       \]
        This means the smaller matrix must have a smaller operator norm. This gives the result.
\end{proof}
                                                                                 
\begin{proposition}\label{propineq6}  Let $\tilde B_Y = \tilde B$ and $\tilde B_X = B.$ Then,
\[
|(S_{p-1}(B) - S_{p-1}(\tilde B), X - Y)^\sharp| \leq
     2(p-1)s(B)^{p-1} \delta(X,Y)^{3/2}.
     \]
\end{proposition}
\begin{proof}
\begin{eqnarray*}
\lefteqn{\left |(S_{p-1}(B) - S_{p-1}(\tilde B), X - Y)^\sharp \right |}  \\
&\leq & \left |(S_{p-1}(B) - S_{p-1}(\tilde B)_X, (X - Y)_X)^\sharp \right |\\
                                       &+& \left | (S_{p-1}(\tilde B),X)^\sharp (X-Y,X)^\sharp\right |.
  \end{eqnarray*}
The second factor is easy to estimate, as $(X-Y,X)^\sharp = 1/2\delta(X,Y)$ and
\begin{eqnarray*}
|(S_{p-1}(\tilde B),X)^\sharp | &=& |\mathcal   Q(\tilde B)^{p-2}(\tilde B,X)^\sharp| \\
 &\leq& s(\tilde B)^{p-2}|(\tilde B,X)^\sharp|\\
 &\leq& s(\tilde B)^{p-1}(\delta(X,Y)(1 + 1/4\delta(X,Y))^{1/2}.
\end{eqnarray*}
In the last inequality we used (\ref{normtilBX}).
Recall $s(\tilde B) \leq s(B)$, so we may replace $s(\tilde B)$ by $s(B)$ and absorb the $(1 + 1/4\delta(X,Y))^{1/2}$ in the constant. Next,
\[
 (S_{p-1}(B) - S_{p-1}(\tilde B)_X, (X - Y)_X)^\sharp =
       (\mathcal Q(B)^{p-2}- \mathcal Q(\tilde B)^{p-2})(B,(X-Y)_X)^\sharp.
\]
From Lemma~\ref{props5}, we can estimate the operator norm of
$\mathcal Q(B)^{p-2}- \mathcal   Q(\tilde B)^{p-2}$ by $(p-2)s(B)^{p-2}\delta(X,Y)$,
 the norm of $B$ by $s(B)$ and the norm of $(X-Y)_X$ by $\delta(X,Y)^{1/2}(1+ 1/4\delta(X,Y))^{1/2}.$
\end{proof}

\begin{proposition}\label{propine7} Let $\tilde B_Y = \tilde B$ and $\tilde B_X = B.$ Then, 
\[
|S_{p/2}(\tilde B)_X - S_{p/2}(B)| \leq p\delta(X,Y) s(B)^{p/2}.
\]
\end{proposition}
This follows in the same fashion as in the first estimate of the previous proposition.
%The next proposition is tedious. We carry out the computations for  a selection of the terms, and indicate where the same arguments can be used to handle the remaining terms. 

\begin{proposition}\label{propine8}  Let 
\[
\tilde B_Y =\tilde B,  \ \tilde B_X = B, \  A_X = A
\]
and $p \geq 4.$ Then,
\[
  |(S_{p-1}(\tilde B) - S_{p-1}(B); B - A )^\sharp| \leq 
         2(p-2)C s(B)^{p-2}\delta(X,Y) |S_{p/2}(B)-S_{p/2}(A)|^{4/p}.
 \]
 Here $C$ is a combinatorial constant independent of $p$ computed using the number of terms in each summand.
\end{proposition}
\begin{proof}  Since $(B-A)_X = B-A$ we can compute
\begin{eqnarray*}
        \lefteqn{|(S_{p-1}(\tilde B) - S_{p-1}(B); B - A )^\sharp|}\\
         &=&| ((\mathcal Q(\tilde B)^{p-2} - \mathcal Q(B)^{p-2})B; B-A)^\sharp \\
         &=& |Tr (\mathcal Q(\tilde B)^{p-2} - \mathcal Q(B)^{p-2}) (B, B-A)^\sharp |\\
        &\leq & 8 \mbox{ max eigenvalue of} \ (Q(\tilde B)^{p-2}-Q(B)^{p-2}) \\
            &\times &  \mbox{largest entry in the $2 \times 2$  matrix} \ (B, B-A)^\sharp.
\end{eqnarray*}
We can take the largest entry in any orthogonal basis (we use the $b_j$).
Here the $2 \times 2$ matrix 
\[
   (B, B-A)^\sharp = \sum_j \beta_j^2 b_j \otimes b_j - \sum_{j,k} 
    \beta_j\alpha_k (B_j, A_k) (b_j \otimes a_k).
\]
We have already estimated the eigenvalues of $Q(\tilde B)^{p-2}-Q(B)^{p-2}$ in Lemma~\ref{props5}. We need only estimate the terms in 
$(B, B-A)^\sharp. $ We recall that $B = \sum_j  \beta_j b_j \otimes B_j$ and $A = \sum_k  \alpha_k a_k \otimes A_k.$  We will use the calculations in the proof of Proposition~\ref{propineq1}, except we now use the bases $b_j $ of $T_xM$ and $B_j$ of $T_XN$ to expand in.  This means that 
\begin{eqnarray*} 
          a_1 &=& \cos\theta b_1 - \sin\theta b_2  \ \ a_2 = \sin \theta b_1 + \cos \theta b_2 \\
          A_1 &=& \cos \phi B_1 - \sin \phi B_2 \ \ A_2 = \sin \phi B_1 + \cos\phi B_2.
\end{eqnarray*} 
We get  that the matrix
$(B, B-A)^\sharp$ is
\begin{eqnarray*} 
        &&  \beta_1(\beta_1 -\alpha_1 E_1 - \alpha_2 E_2)( b_1 \otimes b_1)
       + \beta_2(\beta_2 - \alpha_2 E_1 - \alpha_1 E_2) (b_2 \otimes b_2)\\
          &+& \beta_1 (\alpha_1 E_3 - \alpha_2 E_4) (b_1\otimes  b_2)
          +\beta_2(\alpha_1 E_4-\alpha_2 E_3) (b_2 \otimes b_1).
\end{eqnarray*} 
Here, as before $E_1 = \cos\theta \cos\phi$ and $E_2 = \sin\theta\sin\phi.$ New in this computation is $E_3 = \sin\theta \cos\phi$ and  $E_4 = \cos\theta\sin\phi.$

For convenience, let $Z = |S_{p/2}(B)-S_{p/2}(A)|^2.$ We need to bound all the terms above by a constant times $Z^{2/p}.$ As before, some complications arise if the signs of the cosine and sines are not as expected.  Without changing signs, a computation similar to (\ref{Esecondterm1}) implies 
 \begin{eqnarray}\label{Esecondterm1Y}
      \lefteqn{\left |S_{p/2}(A)-S_{p/2}(B) \right|^2}\nonumber \\
      &=&\sum_j ( \alpha_j^p + \beta_j^p )(1 - E_1 - E_2)
    + E_1 (\alpha_j^{p/2} -  \beta_j^{p/2})^2  \\
&+&E_2 (\alpha_j^{p/2} -  \beta_{j'}^{p/2})^2. \nonumber
 \end{eqnarray}
If $E_1$ and $E_2$ are negative
\[
      Z \geq \sum (\alpha_j^p + \beta_j^p).
      \]
In this case, we can estimate all terms easily.   

In the remaining cases  $E_1$ and $E_2$ are interchangeable by reversing the roles of $b_1$ with $b_2$ and $B_1$ with $B_2$. We may thus assume $E_1 \geq E_2.$  Also
\[
E_1 + |E_2|  <1.
\] 
%
%In the case $E_2 < 0$, we find, since $E_1\geq 0$ and 
%$E_1 + |E_2|  <1$, 
%\[
%         Z \geq \sum (\alpha_j^p + \beta_j^p)(1 - E_1) \geq 
%         \sum   (\alpha_j^p + \beta_j^p)  |E_2|.
%         \]
%Moreover  
%\[
%Z \geq  1/2 \sum_j   (\alpha_j^p+ \beta_j^p) (1 - E_1 -E_2) =
%        1/2 \sum_j (\alpha_j^p +\beta_j^p) (1 - \cos(\phi - \theta)).
%  \]      
%We now return  to $E_2$ with an arbitrary sign.
In the proof of Propositions~\ref{propineq1} and~\ref{propineq2}, we already handled terms that look like the coefficients of $b_j \otimes b_j. $ The only new ingredient we need here is that  $|x (x-y)| \leq |x^{p/2} - y^{p/2}|^{4/p}$ which can be easily checked. 

To handle the off diagonal terms, we first decide which of $|\sin\theta|$ and $|\sin\phi|$ is smaller. Suppose it is $|\sin\theta|$. Then $|E_3|^2 \leq |E_2|. $
We write the coefficient of $ b_1 \otimes b_2$ as
\[
      \beta_1 (\alpha_1 E_3  - \alpha_2 E_4) = 
       \beta_1 (\alpha_1 - \alpha_2)E_3 +\beta_1\alpha_2(E_3 - E_4).
       \]
But 
\begin{eqnarray*}
(E_3 - E_4)^2 &=& \sin(\theta - \phi)^2 = 1 - cos(\theta - \phi)^2 \leq 
        2(1 - cos(\phi - \theta)) \\
        &=&2(1 - E_1 - E_2).  
\end{eqnarray*}
Hence, if  $E_2 \geq 0$ (recall $p\geq 4$)
\begin{eqnarray*}
       |\beta_1 \alpha_1 (E_3 - E_4)|^{p/2}
        &\leq& \beta_1^{p/2} \alpha_1^{p/2} |(E_3 - E_4)|\\
        &\leq& \beta_1^{p/2} \alpha_1^{p/2} 2(1 - E_1 - E_2) \leq 2 Z.
 \end{eqnarray*}
 If $E_2 < 0$ we use instead
\begin{eqnarray*}
Z &\geq & 1/2 \sum_j   (\alpha_j^p+ \beta_j^p) (1 - E_1 -E_2) \\
&\geq& 1/2 \alpha_1^{p/2}\beta_1^{p/2} (E_3 -E_4)^2 \geq 1/2 |\beta_1 \alpha_1 (E_3 - E_4)|^{p/2}.
 \end{eqnarray*}
 For the proof of the first inequality we use $1+E_2>E_1$.

To estimate the other term, we write 
\[
|\beta_1(\alpha_1 - \alpha_2)E_3| 
       \leq \sum_k  \beta_1 |(\beta_1 - \alpha_k)E_3|.
 \]      
 If $E_2 \geq 0$,
      \[
         Z\geq \sum_{j,k} (\beta_j^{p/2} - \alpha_k^{p/2})^2 E_2.
         \]
Recall that we chose $E_2 \leq E_1$ for exactly this purpose. Since $|E_3| ^2 \leq |E_2|$ and we already met the estimate $|x (x-y)| \leq |x^{p/2} - y^{p/2}|^{4/p},$ the bound on this term is complete. The case when $E_2<0$ is easier, since in this case
\[ 
Z \geq \sum (\alpha_j^p + \beta_j^p)(1 - E_1) \geq 
         \sum   (\alpha_j^p + \beta_j^p)  |E_2|.
\]
The coefficient of $b_2 \otimes b_1$ is estimated in the same way. We reverse the roles of $E_3$ and $E_4$ in the case that
 $|\sin\theta| \geq  |\sin\phi|.$
\end{proof}

\subsection{The regularity theorem}
In this section we prove that if $u=u_p$ satisfies the $J_p$-Euler-Lagrange equations, $D(Q(du)^{p/2 - 1}du)=DS_{p/2}(du)$ is in $L^2.$
We find the notation $Q(du)^{q-1}du = S_q(du)$ and $\delta(u,w) = (u-w, u-w)^\sharp$ useful. 
The a priori estimate predicting this is easily obtained from a Bochner formula, but to obtain the result  for a solution only known to be in $W^{1,p}$, we must take difference quotients.  Because we are mapping into  a locally symmetric space, we can locally consider differences obtained using the solution $u$, and the translated solutions $w_t = g_t^*u$, or $w_t(x) = u(g_tx)$ for $g_t = \exp ta$, for a an arbitrary elements in the Lie algebra.   If we can obtain bounds in $L^2$ on the difference quotients,  $1/t \left(S_{p/2}(dw_t)_u - S_{p/2}(du)\right)$, then the covariant derivative of $S_{p/2}(du)$ exists in $L^2.$

Assume $w$ and $u$ are solutions to the $J_p$-Euler-Lagrange equations in a ball, and $\phi$ is a cut-off function with support in the ball $\Omega$ (and equal to 1 on a smaller ball $\Omega'$).   Use the Euler-Lagrange equations for $u$ to get:
\[
<S_{p-1}(du), d(\phi^2(u - w + (u-w,u)^\sharp u))> = 0.  
\]
Rearrange to get
\[
 <S_{p-1}(du),d(\phi^2(u-w))> = 
      -1/2 <\phi^2 (S_{p-1}(du),du)^\sharp \delta(u,w)>.
\]
Write down the same equation with $u$ and $w$ interchanged and add the two equations.  This gives
\[
           <S_{p-1}(dw) - S_{p-1}(du),d(\phi^2(w-u))> = 
           -1/2<\phi^2 Tr(Q(du)^p + Q(dw)^p) \delta(u,w)>.
  \]  
Thus,
\begin{eqnarray*}
    \lefteqn{<\phi^2(S_{p-1}(dw) - S_{p-1}(du)),d(w-u))>}\\
     &=&-< d(\phi^2)(S_{p-1}(dw) - S_{p-1}(du)), w-u>\\ 
     &-& 1/2<\phi^2\delta(u,w) (Tr Q(du)^p + Tr Q(dw)^p)>.
 \end{eqnarray*}
Adding and rearranging some terms, we get the next equation. 

For the purposes of the proof of the next proposition, we label each term in the equation by a Roman numeral. Note that the inner product in $R^{2,1}$ is not positive definite, so to make estimates, we need to project terms into the tangent space at $N$ of either $u$ or $w.$ This explains the elaborate rearrangement of terms:

 \begin{eqnarray*} 
 \lefteqn{<\phi^2(S_{p-1}(dw_u) - S_{p-1}(du)), dw - du> \ (I)}\\
 &=&<\phi^2(S_{p-1}(dw_u) - S_{p-1}(dw)), dw - du>\\
  &+&<\phi^2(S_{p-1}(dw) - S_{p-1}(du)), dw - du> \\
  &=&
    <\phi^2(S_{p-1}(dw_u) - S_{p-1}(dw)), dw_u - du> \ (II) \\
    &+& <\phi^2(S_{p-1}(dw) - S_{p-1}(dw_u)), (dw - du,u)^\sharp u> \ (III)\\
    &- 2&<\phi d\phi(S_{p-1}(dw_u) - S_{p-1}(du)),(w-u)_u> \ (IV)\\
   &-& <d(\phi^2) (S_{p-1}(dw) - S_{p-1}(dw_u)), w-u > \ (V)\\
%   &- 2&<\phi (S_{p-1}(dw_u) - S_{p-1}(dw)); u-w d\phi>  \\
&-1/2& (<\phi^2(\delta(u,w)(Tr Q(du)^p + TrQ(dw)^p)> \ (VI).
\end{eqnarray*}

\begin{proposition}\label{propine9}Assume $w$ and $u$ are solutions to the $J_p$-Euler-Lagrange equations in a ball, and $\phi$ is a cut-off function with support in the ball $\Omega$ (and equal to 1 on a smaller ball $\Omega'$). Then,
 \begin{eqnarray*}  
\lefteqn{1/(4p)<\phi^2 (S_{p/2}(dw_u) - S_{p/2}(du)), S_{p/2}(dw_u) - S_{p/2}(du)>_\Omega }\\
 &\leq& 
64p \max|d\phi|^2  <(s(du)^{p-2}+ s(dw)^{p-2})\delta (w,u)>_\Omega \ (VII)\\
&+ &
 C(p) \max |d(\phi^2)| <s(dw_u)^{p-1}\delta(w,u)^{3/2}>_\Omega \ (VIII)\\
&+& C(p)  <s(dw_u)^p (\phi \delta(w,u))^{p/p-2}>_\Omega \ (IX)\\
 &+&1/2 < \phi^2(Q(dw)^p - Q(du)^p)\delta(u,w)>_\Omega \ (X)\\
   &+&<\phi^2 Q(dw)^p\delta(w,u)^2>_\Omega  \ (XI).
   \end{eqnarray*}
\end{proposition}
\begin{proof}
The proof of this comes from the point wise inequalities. We do not keep track of the middle constants $C(p)$, because they do not appear in the a priori estimates and vanish once we have higher regularity.
\begin{itemize}
\item $(I)$  With the notation from the point wise inequalities, $B = dw_u$ and $A = du,$ we have from Proposition~\ref{propineq1} that
\[
1/p<\phi^2(S_{p/2}(dw_u) - S_{p/2}(du)), S_{p/2}(dw_u) - S_{p/2}(du)> \leq (I).
\]
\end{itemize}
This is equal to four times the left hand side of the inequality in Proposition~\ref{propine9}. We will now write $(I)$ as a sum of the terms $(II) -(VI)$. Below we will estimate these  in terms of $(VII) -(XI)$. Twice we will absorb terms from the right hand side to the left hand side. This explains the coefficient $1/(4p)$ in the left hand side of Proposition~\ref{propine9}.
\begin{itemize}
\item $(II)$  We use the point wise inequality in Proposition~\ref{propine8}. With $\tilde B = dw$, $B = dw_u$ and $A = du,$ we have
 \begin{eqnarray*}  
   \lefteqn{  (|S_{p-1}(dw_u) - S_{p-1}(dw); dw_u - du|)^\sharp }\\ 
   &\leq&    2(p-1)C  s(dw_u)^{p-2} \delta(w,u)|S_{p/2}(dw_u)-S_{p/2}(du)|^{4/p}\\
   &\leq&    2/p  \left(|S_{p/2}(dw_u)-S_{p/2}(du)|^{4/p}\right)^{p/2}\\
   &+&   (p-2)/p   \left(2(p-1)C s(dw_u)^{p-2} \delta(w,u)\right)^{p/(p-2)}\\
&=& 1/{2p}   |S_{p/2}(dw_u) - S_{p/2}(du)|^2 +      
               C  TrQ(dw_u)^p\delta(w,u)^{p/p-2}.
  \end{eqnarray*}
  Multiply by $\phi^2$ and integrate.
The first term can be subtracted from the left hand side  (leaving $1/(2p)$) and the second is found in $(IX)$ of the right hand side of the inequality in Proposition~\ref{propine9}.

\item $(III)$   This is a serious term, containing part of the curvature of $N.$ Note that since $(S_{p-1}(dw_u),u)^\sharp = 0$
and $(du, u)^\sharp=0$,  term $(III)$ is equal to 
\[
(<\phi^2(dw,u)^\sharp;(S_{p-1}(dw),u)^\sharp>).
\]
 Proposition~\ref{morecomp1} bounds this term by 
 \[
 <\phi^2 Q(dw)^p(\delta(w,u)(1  + 1/4\delta(w,u))>.
 \]
This combines with term $(VI)$ to give $(X)$ and $(XI)$ of Proposition~\ref{propine9}.

\item $(IV)$ A direct application of Proposition~\ref{propineq2}  bounds this term by
\begin{eqnarray*}
 \lefteqn{2\max d\phi<\phi |S_{p/2}(dw_u)-S_{p/2}(du)|}\\
 &\times&(\delta(w,u)(1+1/4\delta(w,u)))^{1/2} 2\max (s(du),s(dw_u))^{(p-2)/2}>  \\                                                                           
&\leq& 1/(2p) <\phi^2|S_{p/2}(dw_u)-S_{p/2}(du)|^2> + (VII).
 \end{eqnarray*}
Subtracting this from $(I)$ leaves the $1/(4p)$ as the coefficient on the right hand side of the inequality of Proposition~\ref{propine9}.
\item $(V)$  There is a direct bound of $(V)$ by $(VIII)$ via Proposition~\ref{propineq6}.
 \end{itemize} 
 \end{proof}
We can now finish the regularity theorem.

\begin{theorem}\label{mainregtH}Let $u=u_p$ satisfy the $J_p$-Euler-Lagrange equations
 and  $w = w_t = g_t^*u$, where $g = \exp (at)$ for $a$ an element of the Lie algebra.  Then
\[
\lim_{t \rightarrow 0}1/{t^2} <\phi^2(S_{p/2}(dw_t)_u - S_{p/2}(u)), S_{p/2}(dw_t)_u - S_{p/2}(u)>_\Omega 
\]
 is finite and $S_{p/2}(du) \in H^1(\Omega'). $  Moreover,
\[
      ||S_{p/2}(du)||_{H^1(\Omega')} \leq k p <Q(du)^p>_\Omega^{1/2}
\]
where $k$ depends on $\Omega' \subset \Omega$ but not on $p.$
Moreover,
\[
        ||S_{p/2}(du)||_{H^1(M)} \leq k' p <Q(du)>^{1/2},
  \]
  where $k'$ depends on $M$ but not on $p.$
\end{theorem}

\begin{proof}  We proceed as in Proposition~\ref{propine9} and in an arbitrary neighborhood with arbitrary choice of the Lie algebra element $a.$ Note that $W^{1,p} \subset C^{1 - 2/p}.$ This means that we have a uniform estimate  
  \[           
            \max \delta(w_t,u) \leq kt^{2-4/p}.
\]
Of course, $k$ depends on the manifold and the choice of the Lie algebra element $a. $

We are going to use the estimate from Proposition~\ref{propine9}  to prove Theorem~\ref{mainregtH}.  However, it is not quite correct yet. Formally we need a uniform bound on 
\[
1/t ||{S_{p/2}(dw_t)}_u -S_{p/2}(du)||_{L^2(\Omega')}
\]
 when Proposition~\ref{propine9} is only giving us a bound on
 \[ 
1/t||S_{p/2}({(dw_t)}_u) - S_{p/2}(du)||_{L^2(\Omega')}.
\]
However,
\begin{eqnarray*}
     \lefteqn{||{S_{p/2}(dw_t)}_u -S_{p/2}(du)||_{L^2(\Omega')}}\\
      &\leq&
        ||{S_{p/2}(dw_t)}_u - {S_{p/2}((dw_t)}_u)||_{L^2(\Omega')} +
        ||{S_{p/2}((dw_t)}_u)-S_{p/2}(du)||_{L^2(\Omega')}.
\end{eqnarray*}
A straight forward application of the inequalities of Proposition~\ref{propine7}, with $dw_t= \tilde B$, $dw_u = B$, $Y = w_t(x)$ and $X = u(x)$  bounds
\[
         ||{S_{p/2}(dw_t)}_u - S_{p/2}(({dw_t})_u)||_{L^2(\Omega')} \leq 
        p|| Q((dw_t)_u)^{p/2} \delta(w_t,u)||_{L^2(\Omega')}.
\]
For $p>4$ this goes to 0 at the rate $O(t^{2-4/p}) \leq O(t)$ since $Q(dw_t)^p$ is bounded in $L^1$ and $\delta(w_t,u)  \leq kt^{2(1-2/p)}.$

We start checking the terms on the right of Proposition~\ref{propine9} in reverse order. We need bounds on the order of $t^2.$
\begin{itemize}
\item $(XI)$  $Q(dw_t)^p$ is bounded in $L^1$ and $\delta(w_t,u)^2  \leq k^2t^{4(1-2/p)}.$
This goes to 0 faster than $t^2$ for $p > 4$  and can be neglected.

\item $(X)$   This term looks difficult at first, but we reverse the roles of $w_t$ and add the two inequalities. The left hand side of the inequalities are positive, terms $(X)$ cancel and the other terms are handled the same way in both inequalities. 

\item $(IX)$ The bound is exactly $\delta(w_t,u)^{p/(p-2)}=O(t^2).$ So the contribution to the estimate is exactly bounded by a constant times the $p$ norm of $w_t,$ which is the $p$ norm of $u$ as well. Once we have any estimate on $S_{p/2}\in H^1$, the next corollary and Sobolev embedding imply that this term will converge faster and can be neglected.

\item $(VIII)$ This term is bounded by $||s(dw)^p||_{L^1}^{(p-1)/p} ||d_a u||_{L^p} t^{2 -4/p}$ which goes to 0 faster than $t^2.$ Here we use that $1/t\delta(w_t,u)^{1/2} $ approaches the derivative of $u$ in the direction $a$ as $ t \rightarrow 0.$

\item $(VII)$  This term, when divided by $t^2$, is bounded by
 $2Q(du)^{p-2}(d_a u)^2.$  This is the only term which survives once we know that $u$ is in $C^\alpha$ for $\alpha > 1-2/p.$
 \end{itemize}
\end{proof}

\begin{corollary}\label{cormainredsh} If  $u_p$ satisfies the $J_p$-Euler-Lagrange equations, then
$|du_p|$ is in $L^s$ for all $s.$  Moreover for all $s < \infty$
\[
    ||du_p||_{L^{ps}} \leq 2(k' C_s p)^{2/p}||du_p||_{sv^p}.
    \]
    Here $C_s$ is the norm of the embedding $H^1$ in $L^{2s}$ and $k'$ depends on $M$ and not on $p$.
\end{corollary}
\begin{proof}  By applying the standard inequality $(d|\xi|; d|\xi|)\leq (D\xi; D\xi)^\sharp$ for $\xi=S_{p/2}(du_p)=Q(du_p)^{p/2-1}du_p$,
\[
|d (Tr(Q(du_p)^p)^{1/2}| \leq |D(Q(du_p)^{p/2-1}du_p)|.
\]
   Hence, from Theorem~\ref{mainregtH}  and the Sobolev embedding  $ H^1 \subset L^{2s}$ we have
\begin{eqnarray}\label{eqnsobemb}
            || (TrQ(du_p)^p))^{1/2}||_{L^{2s}} \leq C_s ||(TrQ(du_p)^p)^{1/2}||_{H^1} \leq
            k'C_s p ||du_p||_{sv^p}^{p/2}.
\end{eqnarray}
Since
\[
            |du_p|^p \leq 2^p Tr Q(du_p)^p,
\]
  \[
          ||du_p||_{L^{sp}} \leq 2 (||TrQ(du_p)^p)^{1/2}||_{L^{2s}})^{2/p}.
\]
The result follows directly from this.
 \end{proof}

    Note that minimizing $p$-harmonic maps, $p>2$ are $C^{1, \alpha}$ (cf. \cite{uhlen} and \cite{hardlin}). The higher regularity of the $J_p$-minimizers is an interesting question which we state as a conjecture:
    \begin{conjecture}\label{highregu} Let $u_p: M \rightarrow N$  satisfiy the $J_p$-Euler-Lagrange equations between hyperbolic surfaces, $p>2$. Then $u_p$ is in $C^{1, \alpha}$.
    \end{conjecture}

\section{The limit $q \rightarrow 1$}\label{lipqgoesto1}
Recall from Section~\ref{sect:Varblm} that, as $p \rightarrow \infty$, the minimizers $u_p$ of $J_p$ converge to a best Lipschitz map $u$. In this section we show that, after normalization, the Noether currents of $u_p$ converge converge as well.  More precisely, there exist Radon measures $S$ and $V$ with values respectively in $T^*(M) \otimes E$ and $T^*(M) \otimes ad (E)$ which are weak limits of the (appropriately rescaled) tensors 
 $S_{p-1}(du_p)$ and $V_q= *S_{p-1}(du_p) \times u_p $ associated with the minimizer $u_p$. Similarly, there exist Radon measures $T$ and $W$ with values respectively in $T^*(M) \otimes F$ and $T^*(M) \otimes ad (F)$ which are the weak limits of the (appropriately rescaled)
tensors $T_q = (S_{p-1}(du_p), du_p)^\sharp \rightharpoonup T$ and $W_q= T_q \times id \rightharpoonup V$.   Here $1/p + 1/q = 1.$

We show $V, W$ are closed as a 1-currents with respect to the flat connection on the flat Lie algebra bundles $ad (E)$ and $ad (F)$. 
In Section~\ref{sec7} we will prove  that the supports of these measures  are contained in the canonical  lamination $ \lambda$ associated to the hyperbolic metrics $g$, $h$ and the homotopy class.

\subsection{The normalizations}\label{normkappa} In this section  fix a $2\leq p < \infty$, $1< q \leq 2$ satisfying
(\ref{form:conjugate})
and  let $u_p$  be the $J_p$-minimizer in a fixed homotopy class. 
Choose a normalizing factor $\kappa_p$ and define $U_p= \kappa_p du_p$
such that 
\begin{equation}\label{normintv1}
 ||U_p||_{sv^p}^p= \int_M   Tr  Q(U_p)^p*1 = <Q(U_p)^p> = 1.
\end{equation}
%so that
%\begin{equation}\label{normintv1}
%  \int_M   Tr  Q^p(A_p)*1 =   \delta^{1-p},
%\end{equation}
%let 
%\begin{equation}\label{normintv31}
%  U_p = \deltaA_p \in \Omega^1(ad(E))  
%\end{equation}
%and note
%\begin{equation}\label{normintv312}
%  \ \int_M Tr Q^p(U_p)*1=\kappa_p.
%\end{equation}
Consider the normalized Noether current
\begin{equation}\label{normintv21}
   S_{p-1} =  Q(U_p)^{p-2} U_p  \in  \Omega^1(E).
\end{equation}
%Equation (\ref{eqn:firstvar}) implies that for the induced connection $\nabla_q$ on $E_q$
%\begin{equation}\label{eqn:firstvar2}
%d_{\nabla_q} V_q=0.
%\end{equation}
Note that by (\ref{form:conjugate}) and (\ref{normintv1}) 
\begin{eqnarray}\label{Quv}
 Q(S_{p-1})^2=(S_{p-1} ; S_{p-1})^\sharp=
 Q(U_p)^{2p-2}
\end{eqnarray}
satisfies
\begin{equation}\label{dualnor}
||S_{p-1}||_{sv^q}^q=\int_M Tr Q(S_{p-1})^q*1=\int_M Tr Q(U_p)^p*1=||U_p||^p_{sv^p}=1.
\end{equation}
\begin{lemma} \label{klemma1}Under the normalizations above, 
$\lim_{p \rightarrow \infty} \kappa_p = L^{-1}. $
%(I conjecture that $\kappa_p$ is about $p^{1/p}$ if we ever need it.) 
\end{lemma}
 \begin{proof} 
 By (\ref{normintv1}) and Lemma~\ref{pintconvto}  
\[ 
\lim_{p \rightarrow \infty} \kappa_p^{-1}=\lim_{p \rightarrow \infty}  J_p(u_p)^{1/p} = L.
\]
\end{proof}

From now on we will replace all the Noether currents by the normalized Noether currents. 
Namely set
\[ 
S_{p-1}=S_{p-1}(U_p), \ Z_q = *S_{p-1} \ \mbox{and} \  V_q= Z_q \times U_p.
\]

\subsection{The limit of $V_q$ as $q \rightarrow 1$}
We review the construction of the limiting measures in the case that $N = S^1$. In this case,  $u_p: M \rightarrow S^1$, $S_{p-1} = |du_p|^{p-2}du_p$, and $Z_q = *S_{p-1}=V_q$ is a closed one-form satisfying $d^* |V_q|^{q-2}V_q = 0$. Because  $V_q$ is closed, $V_q = dv_q$ for a local function $v_q$.  If we normalize $V_q$ as described in the previous section, there exists a subsequence $v_q \rightarrow v$ with $V_q = dv_q \rightarrow V = dv$.  Furthermore, $V$ has locally finite total variation and  $v$ is locally a function of bounded variation. 

We could easily extend this to the case when $N$ is flat and of dimension greater than 1.  However, in the present situation, $V$ and $dv$ are one-forms with values in $ad(E)$ and the Killing form is not  positive definite.  In order to resolve this difficulty, we need to make use of the map $u: M \rightarrow N$ in order to project onto the positive definite part. This may look a bit unfamiliar at first. All it means is that we  consider them as tensors with values in the pullback bundle $u^{-1}(TN)$. 

We first concentrate on $S$, although it is not closed, as our support argument is for $S$. To obtain an argument for a closed one form, we translate what we have proved to $V=*S \times u = dv$.  In fact, the two definitions are equivalent.

Let  
\[
f: M \rightarrow H=\tilde M\times_\rho \HH \subset E =\tilde M\times_\rho \R^{2,1}
\]
be a Lipschitz  section and  $\xi \in \Omega^1(E)$. Define $\xi_f \in \Omega^1(E)$ by setting $\xi_f \in \Omega^1(E)$ to be  the section corresponding to the $\rho$-equivariant map $\tilde \xi_{\tilde f}$, where as usual tilde means lift to the universal cover. Recall, from  (\ref{proj1}),  that
$\tilde \xi_{\tilde f} =\tilde \xi+ (\tilde \xi, \tilde f)^\sharp \tilde f.$

\begin{definition}\label{projcurr}
Let  $\gamma$ be a  1-current  with values in  $E$. We write $\gamma_f=\gamma$ if, for any  $\xi \in \Omega^1(E)$, $\gamma[\xi _f]=  \gamma[\xi]$. Note that, because $f$ is Lipschitz, $\gamma[\xi _f]$ is defined.  
\end{definition}

\begin{definition}\label{measrcurr} Let $f$ be a Lipschitz section as before and $\gamma$ a 1-current  with values in  $E$.  
 The {\it mass  of $\gamma$ with respect to $f$} assigns to each open set $U$ of $M$
\[
||\gamma||_{mass,f,U}= \sup \{ \gamma[\xi]: \xi \in \Omega^1(E), spt(\xi) \subset U, \xi=\xi_f, s( \xi) \leq 1 \}.
\]
If $||\gamma||_{mass,f, M}<\infty$  call $\gamma$ a \it{Radon measure with values in $T^*M \otimes E$}.
\end{definition}

Note that, in the above definition, we have used the operator norm ($\infty$-Schatten norm) instead of the $L^\infty$-norm which seems geometrically better suited for our problem. Clearly, this does not affect the notions of ``finite mass" and ``bounded variation" even though the exact values of the norms depend on which norm we are using.

It is convenient to denote 
\begin{equation}\label{def|sp-1}
|S_{p-1}|:=(S_{p-1};U_p)^\sharp = Tr Q(U_p)^p= |U_p|_{sv^p}^p
\end{equation}
viewed as a non-negative  measure on $M$. 
For a test function $\xi \in \Omega^1(E)$, define the 1-current with values in  $E$
\begin{eqnarray}\label{krisn}
S_{p-1}[\xi] =  <S_{p-1},  *\xi >=
         <S_{p-1}, *\xi_{u_p}>= S_{p-1}[\xi_{u_p}].
\end{eqnarray}

 \begin{theorem}\label{TTheorem A5}  Given a sequence $p \rightarrow \infty$, there exists a subsequence (denoted again by $p$), a real-valued positive Radon measure $|S|$ and a Radon measure $S$ with values in $T^*M \otimes E$ such that 
\begin{itemize}
\item $(i)$ $ |S_{p-1}| \rightharpoonup |S| $ and $\int_M |S| = 1$.
\item $(ii)$ $ S_{p-1} \rightharpoonup S$, $||S||_{mass,u,M} \leq 1$ and     $ S=S_u$.  
\item $(iii)$support $S$ in support  $|S|$
\item $(iv)$ $||S||_{mass,u,M}  =1$. In particular $S$ is non-zero.
\end{itemize}  
\end{theorem}
\begin{proof}
For $(i)$ note  that, for a real valued test function $\phi$ with $||\phi||_{L^\infty} \leq 1$,  (\ref{normintv1}) implies
\[
\int_M |S_{p-1}|\phi*1 \leq \big|\big||S_{p-1}|\big|\big|_{L^1}= ||U_p||_{sv^p}^p=1.
\] 
Thus, after passing to a subsequence, there is a non-negative Radon measure $|S|$ such that $ |S_{p-1}| \rightharpoonup |S| $
with $\int_M |S|= \lim_{p \rightarrow \infty} \int_M |S_{p-1}|*1=1$. 

For $(ii)$, let $\xi \in \Omega^1(E)$ be a test function. From~(\ref{krisn}), (\ref{dualnor}) and (\ref{holds}) and the convergence $u_p \rightarrow u$,  
\begin{eqnarray}\label{krisn2}
S_{p-1}[\xi] \leq ||S_{p-1}||_{sv^q} ||\xi_{u_p}||_{sv^p} \leq ||\xi_{u_p}||_{sv^p}.
\end{eqnarray}
We first show that (after passing to a subsequence)   $ S_{p-1} \rightharpoonup S$.
Let $\xi$ be a test function such that the $L^\infty$-norm $ ||\xi||_{L^\infty} \leq 1$ with respect any positive definite metric on $E$.  Because   $u_p$ converges to $u$ in $C^0$, (\ref{krisn2}) implies that $S_{p-1}[\xi]$ is uniformly bounded. The convergence follows.
   For the second statement in $(ii)$,
\begin{eqnarray*}\label{krisn3}
S_{p-1}[\xi] &\leq&  ||\xi_{u_p}||_{sv^p}\\
&\leq&\omega(u_p,u)||\xi||_{sv^p}\\
&\rightarrow& ||\xi||_{sv^p}.
\end{eqnarray*}
The fact that $||S||_{mass,u,M} \leq 1$  follows easily from the above, Lemma~\ref{lemma:normlim} and  the assumption $s(\xi=\xi_u) \leq 1$. In order to check $ S=S_u$,
\begin{eqnarray}\label{curwithvalinbundle}
S_{p-1}[\xi]-S_{p-1}[\xi_u]&=&<S_{p-1}, *(\xi-\xi_u)> \nonumber\\
&=&<S_{p-1},*(\xi-\xi_u)_{u_p}> \nonumber \\
&\leq& ||(\xi-\xi_u)_{u_p})||_{sv^p}\\
&\rightarrow& 0. \nonumber
\end{eqnarray}
 The last holds because $u_p \rightarrow u$ in $C^0$. 
 
 In order to show
$(iii)$, let $B \subset M$ such that $|S| \big |_B=0$. This is equivalent to $|S_{p-1}| \rightarrow 0$ in $L^1(B)$. Formula~(\ref{Quv}) and H\"older imply that $||S_{p-1}||_{L^1} \rightarrow 0$ on $B$. Hence, for any test function $\xi$
 with support in $B$,
 \begin{eqnarray*}
S_{p-1}[\xi] \leq K ||S_{p-1}||_{L^1} ||\xi_{u_p}||_{L^\infty} \rightarrow 0.
\end{eqnarray*}

 Finally we are going to show $(iv)$. 
 We already know from $(ii)$, 
$||S||_{mass,u} \leq 1$.  
Conversely,  let $w_k \rightarrow u$ be a smooth Lipschitz approximation to $u$ as in Theorem~\ref{thmgreenewu} and set
\begin{eqnarray}\label{kaly479}
\xi_k = h_kdw_k \ \mbox{where} \ h_k = s((dw_k)_u)^{-1}.
\end{eqnarray}
Let $Z=\lim_{q \rightarrow 1} Z_q=*S$. Since $S_u=S$,
\begin{eqnarray}\label{kaly4}
 (Z +d\psi)[(\xi_k)_u ]
  = h_k Z [dw_k ]+ h_k<d\psi, *(dw_k)_u>.
\end{eqnarray}
Note that
\begin{eqnarray}\label{kkaly6}
<d\psi, *(dw_k)_u> \nonumber  
&=& <d\psi, *(dw_k +(dw_k,u)^\sharp u)>\\
&=& <d\psi, *dw_k>+<d\psi, (*dw_k, u-w_k)^\sharp u>. 
  \end{eqnarray}
  In the last equality we used the fact that $dw_k$ is perpendicular to $w_k$.
  The Lipschitz bound on $w_k$, the fact that $\psi$ is of bounded variation plus  $w_k \rightarrow u$ in $C^0$ imply that the second term limits on 0.
The first term is zero from $d^2=0$ because $\psi$ is an actual section of the flat bundle $E$. Thus, 
\begin{eqnarray}\label{kalyps6}
\lim_{k \rightarrow \infty}(Z+d\psi)[(\xi_k)_u]=\lim_{k \rightarrow \infty}h_k Z[dw_k].
  \end{eqnarray}

Note that $\liminf_{k \rightarrow \infty} h_k  \geq 1/L$ since 
\[
 \limsup_{k \rightarrow \infty} s(({dw_k})_u) \leq \limsup_{k \rightarrow \infty}( \omega(u,w_k)s(w_k) ) = L. 
\]
As a last step, 
\begin{eqnarray}\label{kaly1} 
          Z[dw_k] = \lim_{ p\rightarrow  \infty} <*S_{p-1}, *(dw_k)> \nonumber \\
           = \lim_{ p\rightarrow  \infty} <Q(U_p)^{p-2}U_p, dw_k>.
 \end{eqnarray}
But according to Proposition~\ref{supptprop5}, with $w = u_p$ and $f = w_k$
\[
             <  \omega(u_p, w_k)Q(U_p)^{p-2}U_p, du_p - dw_k> = 0.
 \]
Thus,
\begin{eqnarray}\label{kaly2}
      \lim_{ p\rightarrow  \infty}<Q(U_p)^{p-2}U_p, dw_k> &=&
       \lim_{ p\rightarrow  \infty} \left(1/\kappa_p<\omega(u_p,w_k)Q(U_p)^p>\right) \nonumber\\
       & \geq&  \lim_{ p\rightarrow  \infty} 1/\kappa_p
       = L.
 \end{eqnarray}
 The last limit follows from Lemma~\ref{klemma1}.
Since  $\liminf_{k \rightarrow \infty} h_k  \geq 1/L$,  (\ref{kalyps6}), (\ref{kaly1}) and (\ref{kaly2}) imply $||S||_{mass,u,M}=||Z||_{mass,u,M} \geq1$. This  completes the proof. 
%({\bf{I don't see what this does:}} Using a partition of unity, one can show convergence in measure in neighborhoods by using projection into a constant function $f$ identically $X.$ In the flat Euclidean norm
%\[ 
%<W_1,W_2>_X = <{W_1}_X,{W_2}_X>^\sharp + <W_1,X>^\sharp <W_2,X>^\sharp
%\] 
%we can get convergence in measure.) 
\end{proof}

%Results (i)-(vi) are a straightforward application of the above estimates.     The estimate  (vi) follows from Lemma~\ref{LLemma A4}. We now return to the original chapter to prove (vii).
We have  analogues of Definition~\ref{measrcurr} and Theorem~\ref{TTheorem A5} for $V_q=*(S_{p-1}\times u_p)$ as follows:

\begin{definition}\label{measrcurr2} Let $f$ be a Lipschitz section as before and $\zeta$ a 1-current  with values in the flat bundle $ad(E)$. 
 For $U \subset M$ open, define
\begin{eqnarray*}
||\zeta||_{mass,f,U} &=& \sup \{ \zeta[\xi \times f]: \xi \in \Omega^1(E), spt(\xi) \subset U,  s( \xi_f) \leq 1 \}\\
&=& \sup \{ \zeta[\phi]: \phi \in \Omega^1(ad(E)), spt(\phi) \subset U, \phi=\phi_f, s( \phi) \leq \sqrt2  \}.
\end{eqnarray*}
If $||\zeta||_{mass,f,M}<\infty$  we call $\zeta$ a \it{Radon measure with values in $T^*M \otimes ad(E)$}.
\end{definition}

 \begin{theorem} \label{thm:limmeasures0} Given a sequence $p \rightarrow \infty$ ($q \rightarrow 1$) there exists a subsequence (denoted again by $p$) and a  Radon measure $ V$ with values in $T^*M \otimes ad(E)$ and support equal to the support of $S$ such that  $V_q \rightharpoonup  V.$ Furthermore,  $V$ is a closed 1-current,
  $V=V_u$ (i.e $V [\phi]=V [\phi_u]$ for all test functions $\phi$) and $||V||_{mass,u,M} =2$. In particular $V$ is non-zero.
 \end{theorem}
 \begin{proof}
 Recall  $Z_q=*S_{p-1}$ and $V_q =Z_q \times u_p$. We first claim:
\begin{itemize}
 \item $(i)$ $Z_q = V_q u$
\item $(ii)$ $V_q  \rightharpoonup  V = Z \times u$ if and only if $Z_q  \rightharpoonup  Z$ as $q \rightarrow 1$. 
\item  $(iii)$ $Z = V u $ 
\item $(iv)$ the supports of $Z $ and $V$ are identical.
\end{itemize}
Indeed, in Section~\ref{sect:regtheor} we proved that $S_{p-1} $ is in $L^s$ for all $s$, which gives us enough regularity to make the algebraic computations that follow. The linear algebraic relationships are valid if $u$ is merely continuous and we are assuming it is Lipschitz.  Multiplying functions in $L^s$ and  measures by a continuous tensor preserves either $L^s$ or the measure space.  Hence $(i)$-$(iv)$ are straightforward.  

It follows that $V_q \rightharpoonup  V$, and that the support of $V$ is equal to the support of $S$. Also, since $V_q$ is closed, $V$ is also closed.
Finally, $||V||_{mass,u,M} =2$
 follows from extending these linear identities relating $Z $ and $V$ to identities relating the test functions $\xi$ (for $Z)$ and $\phi$ (for $V$). In other words, let
 $\phi =  \xi \times u$. By Proposition~\ref{helpemb}$(iii)$, $(\phi, \phi)^\sharp=2(\xi,\xi)^\sharp$.
So $s(\phi) = s(\xi)$.
But also for the same reason  
\[
          V [\phi] =  (Z \times u)[\xi \times u] = 2 Z[\xi]. 
\]
This yields $||V||_{mass,u,M} =2$ immediately. 
\end{proof}
 
% \begin{remark}\label{expmass} {\bf CHANGE}
%Our definition of mass (and therefore also of least gradient) may at first seem a bit strange: However, the only difference is that we use on the test functions the $\infty$-Schatten norm instead of the $L^\infty$-norm. (Note that already in  \cite{simon} it is acknowledged that there is nothing natural about the $L^\infty$-norm  and other norms could be used instead.) The conditions we impose $\xi=\xi_u$ and $\phi=\phi_u$ mean that we only consider sections of $u^{-1}(TN)$ as test functions. This is the right definition, because we have already shown that $S=S_u$ and $V=V_u$ and thus we would like to compute their mass as currents of $u^{-1}(TN)$. The same for least gradient. If we were looking at the local problem of maps between domains in Euclidean space this condition would be unnecessary. We will further explore these properties elsewhere.
%\end{remark}

 We have seen that the 1-currents $V_q$ and $V$ are closed. Therefore, by passing to the universal cover, or working locally, we can construct Lie algebra valued functions $v_q \rightarrow v$ weakly in $BV_{loc}$ and $dv_q=V_q$, $dv=V$. The following theorem follows  from Theorem~\ref{thm:limmeasures0}.
 
 \begin{theorem} \label{thm:limmeasures04r} There exists a local Lie algebra valued function of bounded variation $v$ such that 
 $dv=V$. 
  \end{theorem}
 It is worth mentioning that, even though $dv=(dv)_u$, $v$ is not equal to $v_u$. 
 
 \subsection{The limit of $W_q$ as $q \rightarrow 1$}\label{limq----1} 
We continue with the notation of the previous section. We write $||.||_{mass}=||.||_{mass,id}$. All the tensors are rescaled. 
 
\begin{theorem}\label{thm:limmeasures450}  Given a sequence $p \rightarrow \infty$ ($q \rightarrow 1$), there exists a subsequence (denoted again by $\{p\}$) 
%a real-valued positive Radon measure $|S|$  a closed 1-current $*S$ with values in $E$ 
and  distributions $T$ and $W$ with values in $Sym^2(T^*M)$, $T^*M \otimes ad(F)$ respectively such that after normalizing 
\begin{itemize}
\item $(i)$  $ T_q \rightharpoonup T$, $ W_q \rightharpoonup W$ 
\item $(ii)$ $dW=0$  with respect to the flat connection on $ad(F)$ and  $W_{id}=W$
\item $(iii)$ $Tr_gT=|S|$ and $*(\omega_{mc}\wedge W)^\sharp=2|S|$
\item $(iv)$The supports of  $T$, $W$ and $|S|$ are equal and contained in the canonical geodesic lamination $ \lambda$ associated to the hyperbolic metrics $g$, $h$ and the homotopy class.
%\item $(v)$  $||T||_{mass,M}=1$ and $||W||_{mass,M}=2$ as measures. 
\end{itemize}
%\begin{itemize}
%%\item $(i)$ $ |S_{p-1}| \rightharpoonup |S| $ where  $\int_M |S|*1 = 1$
%\item $(ii)$    $ V_p \rightharpoonup V$
%\item $(iii)$ the measure $S$ and the currents $|S|$ and $V$ have the same support
%\item $(iv)$ The closed current $*S$ and $V$ are mass minimizing with respect to the best Lipschitz map $u$.
%\end{itemize}  
%\end{theorem}
%Moreover,
%\begin{theorem}\label{thm:supptmeasure} The support of the currents $S$ and $V$  is contained in the canonical geodesic lamination $ \lambda$ associated to the hyperbolic metrics $g$, $h$ and the homotopy class.
\end{theorem}

\begin{proof} Let   $T_q=T'_q-1/pTrQ(U_p)^pg$ where 
$T'_q=(S_{p-1}, U_p)^\sharp=(Q(U_p)^{p-2}U_p, U_p)^\sharp$.
From Theorem~\ref{thm:limmeasures0}, $|S_{p-1}|=TrQ(U_p)^p$ converges to a nonnegative Radon measure $|S|$ and thus  $T_q$ and $ T'_q$ have the same limit. For a test function $\xi=\xi_1 \otimes \xi_2$  of  $TM \otimes TM$ supported in $U$
\begin{eqnarray*}
|T'_q(\xi)|&=&|<S_{p-1}\xi_1, U_p \xi_2>| \leq ||S_{p-1}||_{sv^q} ||U_p||_{sv^p}||\xi||_{sv^\infty}\\
&\leq& ||\xi||_{sv^\infty}.
\end{eqnarray*}
Here we used (\ref{dualnor}). It follows that $ T'_q \rightharpoonup T$.  The proof $ W_q \rightharpoonup W$ is similar. This proves $(i)$.  
$(ii)$ Follows because $W$ is a weak limit of distributions $W_q$ such that $(W_q)_{id}$=$W_q$, $dW_q =0$ and 
$(iii)$ follows from   Proposition~\ref{positiv1}, 
by taking weak limits.

We now come to $(vi)$.  By $(iii)$, the support of  $|S|$ is contained in the support of $T$. To see the converse,
denote the entries of the matrix $Q^l(du_p)$ by $Q^l_{ij}$. In normal coordinates,  
 \begin{eqnarray}\label{tensorS0}
 Q^2_{ij}(du_p)= d_\alpha u_p^id_\alpha u_p^j, \ \  T'_{q, \alpha \beta} = \kappa_p^pQ^{p-2}_{ij}d_\alpha u_p^i d_\beta u_p^j.
\end{eqnarray}
Thus, using the symmetry of $Q$,
\begin{eqnarray*}
(T'_q;T'_q)&=&\kappa_p^{2p}Q^{p-2}_{ij}d_\alpha u_p^i d_\beta u_p^jQ^{p-2}_{kl}d_\alpha u_p^k d_\beta u_p^l \\
&=&\kappa_p^{2p}Q^{p-2}_{ij} d_\alpha u_p^id_\alpha u_p^k Q^{p-2}_{kl}d_\beta u_p^ld_\beta u_p^j\\
&=&\kappa_p^{2p}Q^{p-2}_{ji} Q^2_{ik} Q^{p-2}_{kl}Q^2_{lj}\\
&=&\kappa_p^{2p}Q^p_{jk}  Q^p_{kj}\\
&=&TrQ(U_p)^{2p}.
\end{eqnarray*}
It follows that on  any ball $B$ the $L^1$-norm of $T'_q$ is bounded above and below by a constant times the $L^1$-norm of $|S_{p-1}|=TrQ(U_p)^p$. This in turn implies that the support of $T'$ is contained in the support of $|S|$. Indeed, choose a  ball $B \subset M$ that misses the support of $|S|$. Then, $|S_{p-1}| \rightharpoonup 0$ which implies, integrating against the function 1, that the $L^1$-norm of  $|S_{p-1}|$ (and hence also the  $L^1$-norm of  $|T'_q|$) converges to zero. Therefore,  $B$ misses the support of $|T'|$.

%$||T'_q||_{L^1} \rightarrow 0$ which, by the claim above, is equivalent to $\big|\big||S_{p-1}|\big|\big|_{L^1} \rightarrow 0$ i.e $B$ misses the support of $|S|$. 
Since $W_q=\beta(T_q)=T_q \times id$ and $\beta$ is an isometry (up to the constant factor of  $\sqrt 2$, cf. \cite[Proposition 3.5(iii)]{daskal-uhlen2}) the support of $T$ is equal to the support of $W$. 
By Theorem~\ref{thm:limmeasures0}  the supports are contained in the canonical lamination. 
%We finally prove $(iv)$. We already know that $||T||_{mass,M} \leq 1$ and $||W||_{mass,M}\leq 2$. By~$(iii)$, $*(\omega_{mc}\wedge W)^\sharp=2Tr_gT=2|S|$ and thus, combining with Theorem~\ref{thm:limmeasures0}$(ii)$,
%\[
%W(\omega_{mc})=2 \int_M |S|*1=2. 
%\]
%This shows $||T||_{mass,M} \geq 1$ and $||W||_{mass,M}\geq 2$ and completes the proof.
\end{proof}

As with the case of $V_q$ and $V$, we can construct local Lie algebra valued functions  $w_q,w$ such that  $dw_q=W_q$ and  $dw=W$. The function $w$ is locally of bounded variation and $w_q \rightharpoonup w$ weakly in $BV_{loc}$. We will explore the global properties of $v$ and $w$ in our next paper \cite{daskal-uhlen2}. As it turns out these local Lie algebra functions induce a transverse measure on the canonical lamination.

 \section{The support of the measures}\label{sec7}

 \subsection{The support theorem}The main result of this section is to provide a proof  of the following theorem:

\begin{theorem}\label{thm:supptmeasure} The support of the measure $|S|$ is contained in the canonical geodesic lamination $ \lambda$ associated to the hyperbolic metrics $g$, $h$ and the homotopy class (cf. Definition~\ref{canlamin}).
\end{theorem} 

For the rest of this section, we rescale $M$ to assume $L = 1.$  This rescales  the curvature of $M$ to be $-1/L^2$ and multiplies the volume by $L^2.$  We recall that $L > 1$ originally, so $M$ becomes a hyperbolic manifold of larger volume and smaller curvature, although we do not need this in this section.  While this is not strictly necessary, it allows us to avoid carrying the extra factor of $L$ around.

From now on we assume that $u_p$ satisfies the $J_p$-Euler-Lagrange equations, $W =du_p$ and $f$ is a comparison best Lipschitz map.
This means that $s(df) = s((df)_f)\leq L$. From Proposition~\ref{supptprop5} 
\[
   < \omega(u_p,f) S_{p-1}(W), W - F> = 0
\]
where $F=\omega(u_p,f)^{-1}(df)_{u_p}$ and, from  Corollary~\ref{compprojnorm},  $s((df)_{u_p}) \leq \omega(u_p,f) s(df).$ 
Recall that in (\ref{deromea}) we have defined $1 + (u_p-f, u_p-f)^\sharp = \omega(u_p,f) \geq 1$ as a weight, which now depends on $x \in M$. We will find it both useful and a nuisance  in the rest of the section.

  Note that $F_{u_p} = F$ and $(df)_f = df.$ In the rest of the section, we will use this fact without comment to obtain inequalities  in tangent spaces where the metric is positive definite which are not available in $\R^{2,1}.$

In order to localize the  Euler-Lagrange equations above, we check on the region in which the integrand is non-negative.  Let 
\[
Y_p =Y_p(F) = \{x \in M: (S_{p-1}(du_p); du_p - F)^\sharp \geq 0 \}.  
\]
Another way to put this is to  let $H = (S_{p-1}(du_p); du_p - F)^\sharp$.  Now $H$ is an $L^1$ function.  We can define $Y_p$ is the support of $|H|-H.$  
%Alternatively, as we commented earlier,  we can approximate by smooth maps ${\hat u}_p$ and $ \hat f$, make the estimates with an error term in the Euler-Lagrange equations, and then take the limit as the approximations limit on $u_p$ and $f.$

\begin{proposition}\label{supptprop8} Let $\kappa_p$  be the normalization introduced in (\ref{normintv1}) and $U_p = \kappa_p du_p. $ Assume that $s(F) \leq 1 (=L).$ If $Y^c_p=M \backslash Y_p$, then  
\[
\lim_{p \rightarrow \infty} <S_{p-1}(U_p), F - du_p>_{Y^c_p}  = 0. 
\]
\end{proposition}
\begin{proof} Write the integrand as
\[ 
        (Q(U_p)^{p-2}U_p; F- \kappa_p du_p)^\sharp  
         + (1-\kappa_p^{-1}) TrQ(U_p)^p.
 \]
By the normalization we have imposed, and  $\kappa_p \rightarrow 1$, the limit of the integral of the second term is 0.  
By Lemma~\ref{sptlemma7}, with $\omega$ to be the characteristic function of $Y^c_p$ and $W=du_p$,
\begin{eqnarray*}
      <Q(du_p)^{p-2}du_p, F -du_p>_{Y^c_p} &\leq& 
      1/p \left(<Q(F)^p>_{Y^c_p} - <Q(du_p)^p>_{Y^c_p}\right) \\
      &\leq& 2/p<s(F)^p>  \\
     &\leq& 2/p \ vol(M).
\end{eqnarray*}
 In the second inequality we used   (\ref{shatl1}).  
\end{proof}

\begin{support}  Let $f$ be any comparison best Lipschitz map and choose a ball $B \subset M$ away from the maximum stretch locus $\lambda_f$ of $f$. Normalize as before, $<Q(U_p)> = 1$. By Proposition~\ref{supptprop5}, 
\[
 <\omega(u_p,f)S_{p-1}(du_p), du_p - F> = 0
 \]
  where $F = \omega(u_p,f)^{-1}(df)_{u_p} $ satisfies $s(F) \leq s(df)$ and $\omega(u_p,f) \geq 1$.
 Hence, the following  is true for integrals with positive integrands:
\begin{eqnarray*}
<Q(U_p)^{p-2}U_p, du_p - F>_{B \cap Y_p} &\leq& <Q(U_p)^{p-2}U_p, du_p - F>_{Y_p} \\ 
   &\leq&         <\omega(u_p,f)Q(U_p)^{p-2}U_p, du_p - F)>_{Y_p} \\
   &=&<\omega(u_p,f)Q(U_p)^{p-2}U_p,  F - du_p>_{Y^c_p}\\
       &\leq& k<Q(U_p)^{p-2}U_p, F - du_p>_{Y^c_p}\\
       &\rightarrow& 0.  
\end{eqnarray*}
In the third line we used the Euler-Lagrange equations. Here $k = \max \omega(u_p,f)$ is finite since $u_p\rightarrow u$ in $C^0.$ 
The last term converges to zero by Proposition~\ref{supptprop8}.

Multiplying by $\kappa_p \rightarrow 1$,   
\[
         \lim_{p \rightarrow \infty}<Q(U_p)^{p-2}U_p, U_p - \kappa_p F>_{B \cap Y_p} = 0.
\]
By the definition of $Y_p$,
\[
       <Q(U_p)^{p-2}U_p, U_p - \kappa_p F>_{B \cap Y^c_p} \leq 0.
\]
By combining with the previous equality,
\begin{eqnarray}\label{tauhat21}
 \lim_{p \rightarrow \infty} <Q(U_p)^{p-2}U_p,  U_p -\kappa_p F>_B \leq 0.
\end{eqnarray}
Choose $p$ large enough so that 
\begin{eqnarray}\label{tauhat}
\tau < \kappa_p s(F) < \hat \tau < 1.
\end{eqnarray}
Then rewrite  (\ref{tauhat21}) as 
\begin{eqnarray}\label{tauhat2132}
  \lefteqn{\lim_{p \rightarrow \infty} ((1 - \hat \tau^{1/2})<Q(U_p)^p>_B}  \nonumber \\
 &-& \hat \tau^{1/2}<Q(U_p)^{p-2}U_p, \kappa_p \left(F/\hat \tau^{1/2}\right) - U_p>_B  \leq 0.
\end{eqnarray}
By Lemma~\ref{sptlemma7}  with $\omega$ equal to the characteristic function of $B$, discarding the  negative term $- 1/p <Q(U_p)^p>_B$ and noting (\ref{tauhat}),
\begin{eqnarray*}
<Q(U_p)^{p-2}U_p, \kappa_p \left(F/\hat \tau^{1/2}\right)  - U_p>_B &\leq&
1/p< Q(\kappa_p F /\hat \tau^{1/2})^p>_B\\
&\leq& 2/p <{{\hat \tau}^{p/2}}> \\
&\leq& \left(2/p\right) {\hat \tau}^{p/2}  vol(M)\\
 &\rightarrow& 0.
\end{eqnarray*}
It follows from (\ref{tauhat}) and (\ref{tauhat2132}) that $\lim_{p \rightarrow \infty} |S_{p-1}|(B)=\lim_{p \rightarrow \infty} <Q(U_p)^p>_B = 0.$ Thus $B$ misses the support of $|S|$ and completes the proof.
\end{support}

\begin{corollary}\label{cor:supptmeasure}The supports of  the measures $S$,$V$, $T$ and $W$
are contained in the canonical lamination.
\end{corollary}

\begin{proof} Follows  from  Theorem~\ref{thm:supptmeasure}, Theorem~\ref{TTheorem A5}, Theorem~\ref{thm:limmeasures0} and Theorem~\ref{thm:limmeasures450} 
\end{proof}

\begin{remark}The proof of Theorem~\ref{thm:supptmeasure} and therefore also Corollary~\ref{cor:supptmeasure} does not use the fact that the Lipschitz constant is greater than one. Without this assumption the proof implies that the supports of the measures are contained in the set $\lambda=\cap_{f \in \mathcal F} \lambda_f$ (cf. Definition~\ref{canlamin}). Since the measures are not zero, it also follows that $\lambda$ is non-empty. However $\lambda$ may not be a geodesic lamination (cf. \cite{kassel}).
\end{remark}

\end{document}